\documentclass[10pt]{amsart}

\usepackage{amsfonts,amssymb,amsmath,amscd,amstext}
\usepackage[colorlinks=true,linkcolor=blue,citecolor=blue]{hyperref}
\usepackage[utf8]{inputenc}
\usepackage{microtype}
\usepackage{graphicx}
\usepackage{changes}

\renewcommand{\le}{\leqslant}
\renewcommand{\ge}{\geqslant}
\newcommand{\ptl}{\partial}

\newcommand{\rr}{{\mathbb{R}}}

\newcommand{\nn}{{\mathbb{N}}}

\newcommand{\la}{\lambda}

\newcommand{\eps}{\varepsilon}

\newcommand{\ga}{\gamma}
\newcommand{\Ga}{\Gamma}

\newcommand{\escpr}[1]{\langle#1\rangle}
\newcommand{\co}{\cos(\theta)}
\newcommand{\si}{\sin(\theta)}
\newcommand{\kk}{\kappa}
\newcommand{\mh}{\mathcal{H}}

\newtheorem{theorem}{Theorem}[section]
\newtheorem{proposition}[theorem]{Proposition}
\newtheorem{lemma}[theorem]{Lemma}

\theoremstyle{definition}

\newtheorem{remark}[theorem]{Remark}
\newtheorem{example}[theorem]{Example}

\newtheorem{definition}[theorem]{Definition} 
\numberwithin{equation}{section}
\theoremstyle{remark}

\newcommand{\wdw}{\wedge \cdots \wedge}

\begin{document}

\title[Higher dimensional holonomy map for ruled submanifolds]{Higher dimensional holonomy map for ruled submanifolds in graded manifolds}

\author[G.~Giovannardi]{Gianmarco Giovannardi}
\address{Dipartimento di Matematica, Piazza di Porta S. Donato 5, 401
26 Bologna, Italy}
\email{gianmarc.giovannard2@unibo.it}

\date{\today}

\thanks{The author has been supported by Horizon 2020 Project ref. 777822: GHAIA,  MEC-Feder grant MTM2017-84851-C2-1-P and PRIN 2015 ``Variational and perturbative aspects of nonlinear differential problems” }

\begin{abstract}
The deformability condition for submanifolds of fixed degree 
immersed in a graded manifold can be expressed as a system of first order PDEs. 
In the particular but important case of ruled submanifolds, 
we introduce a natural choice of coordinates, which allows to deeply simplify the formal 
expression of the system, and to reduce it to a system of ODEs along a characteristic direction. 
We  introduce a notion of higher dimensional holonomy map in analogy with the one-dimensional case \cite{MR1189496}, and we provide a characterization for singularities  as well as a deformability criterion.
\end{abstract}

\maketitle

\thispagestyle{empty}

\bibliographystyle{abbrv} 

\tableofcontents
\section{Introduction}
The goal of this work is to study the deformability of a some particular kind of submanifolds immersed in an equiregular graded manifold $(N,\mh^1,\ldots,\mh^s)$, that is a smooth manifold endowed with a filtration of sub-bundles of the tangent bundle $\mh^1\subset\mh^2\subset\cdots\subset\mh^s=TN$ satisfying $[\mh^i,\mh^j]\subset \mh^{i+j}$, $i,j\ge 1$. 

Given $p\in N$, a vector $v\in T_pN$ has degree $i$ if $v\in  \mh_p^i $ but $v \notin \mh_p^{i-1}$.
When we consider an immersed submanifold $\Phi:\bar{M} \to N$ and we set $M=\Phi(\bar{M})$, the interaction between the tangent space $T_{\bar{p}} M=(d\Phi)_{\bar{p}}(T_{\bar{p}} \bar{M})$, where $(d\Phi)_{\bar{p}}$ denotes the differential of $\Phi$ at $\bar{p}$, and the filtration $\mh_p^1\subset \mh_p^2\subset\cdots\subset\mh_p^s$ is embodied by the induced tangent flag
\begin{equation}
\label{eq:intanflag}
 T_{\bar{p}} M \cap \mh_p^1 \subset \cdots \subset T_{\bar{p}} M \cap \mh_p^s,
\end{equation}
where $p=\Phi(\bar{p})$, $\bar{p} \in \bar{M}$. The smooth submanifold $M$ equipped with the induced filtration  pointwise described by \eqref{eq:intanflag} inherits a graded structure, that is no more equiregular. M. Gromov in \cite{Gromov} consider the  homogeneous dimension of the tangent flag \eqref{eq:intanflag} to define the pointwise degree by 
\[
\deg_M(\bar{p})= \sum_{j=1}^s j (\tilde{m}_j- \tilde{m}_{j-1}),
\]
where  $\tilde{m}_{0}=0$ and $\tilde{m}_j=\text{dim}(T_{\bar{p}}M \cap \mathcal{H}_p^j )$. In an alternative definition provided in \cite{vittonemagnani},  the authors write the $m$-tangent vector to $M= \Phi(\bar{M})$ as linear combination of simple $m$-vectors $X_{j_1} \wdw X_{j_m}$ where $(X_1,\ldots,X_n)$ is an adapted basis of $TN$, see \cite{Bellaiche} or \eqref{local adapted basis to the bundle}. Then the pointwise degree is the maximum of the degree of the simple $m$-vectors whose degree is in turn given by the sum of the degrees of the single vectors appearing in the wedge product.  The degree $\deg(M)$ of a submanifold $M$  is the maximum of the pointwise degree among all points in $\bar{M}$. 

In \cite{vittonemagnani} V. Magnani and D. Vittone introduced a notion of area for submanifolds immersed in Carnot groups 
that later was  generalized by  \cite{2019arXiv190505131C} for immersed submanifolds in graded structures.
Given a Riemannian metric $g$ in the ambient space $N$, the area functional $A_d(M)$ in  \cite{2019arXiv190505131C} is obtained by a limit process involving the Riemannian areas of $M$ associated to a sequence of dilated metrics $g_r$ of the original one $g$. The density of this area is given by the projection of the $m$-vector  $e_1\wedge \ldots \wedge e_m$ tangent to $M$ onto the space of $m$-vectors of degree equal to $d=\deg(M)$, see equation \eqref{eq:projected_integral_formula_Ad}. The central issue is that the area functional depends on the degree $\deg(M)$ of the immersed submanifold $M$. Thus, if we wish to compute the first variation formula for this area functional we need to deform the original submanifold by variations $\Ga(\bar{p},\tau)$ that preserve the original degree $\deg(M)$. This constraint on the degree gives rise to a first order system of PDEs  that defines the \emph{admissibility} for vector fields on $M$.

The simplest example of immersion is given by a curve $\ga:I \subset \rr \to N$, with $\ga'(t)\neq0$ at every $t\in I$. The pointwise degree of  $\ga(I)$ at $\ga(t)$ is the degree of its tangent vector $\ga'(t)$ at every $t\in I$. In this particular case the admissibility system is a system of ODEs along the curve $\ga$.  This restriction on vector fields produces the phenomenon of \emph{singular curves}, that do not admit enough compactly supported variations in the sub-bundles  determined by the original degree of $\ga$. This issue has been addressed by L. Hsu in \cite{MR1189496} and R. Bryant and L. Hsu in \cite{MR1240644}.
These two works are based on the Griffiths formalism \cite{MR684663} that studies variational problems using the geometric theory of exterior differential system \cite{MR924805,MR1083148} and the method of moving frames developed by E. Cartan \cite{MR1190006}.
In Carnot manifolds $(N,\mh)$, that are a particular case of graded manifolds where the flag of sub-bundles is produced by a bracket generating distribution $\mh$, the usual approach to face this problem is by means of the critical points of the endpoint map \cite{Montgomery}. The presence of singular curves is strongly connected with the existence of \emph{abnormal geodesics}, firstly established by R. Montgomery in \cite{Mont94a,Mont94b}. In the literature many papers concerning this topic have been published, just to name a few we cite \cite{AgrachevSachkov, AgrachevBarilariBoscain, LeonardiMonti, Extremalcurves, Monti, AgracevSarychev,Rifford}.  The paper \cite{LeonardiMonti} by E. Le Donne, G.P. Leonardi, R. Monti and D. Vittone is specially remarkable because of the new algebraic characterization of abnormal sub-Riemannian extremals in stratified nilpotent Lie groups.

 More precisely,  L. Hsu \cite{MR1189496} defines the singular curves as the ones along which the \emph{holonomy map} fails to be surjective. This holonomy map studies the controllability along the curve restricted to $[a,b]\subset I$ of a system of ODEs embodying the constraint on sub-bundles determined by the degree. In \cite[Section 5]{2019arXiv190204015C} the authors revisited this construction and  defined an \emph{admissible}  vector field as a solution of this system. A powerful characterization of singular curves in terms of solutions of ODEs is given by \cite[Theorem 6]{MR1189496}. On the other hand, when a curve $\ga$ is regular restricted to $[a,b]$, \cite[Theorem 3]{MR1189496} ensures that for any compactly supported admissible vector field $V$ on $[a,b]$ there exists a variation, preserving the original degree of $\ga$, whose variational vector field is $V$. Then, only for regular curves this deformability theorem allows us to compute the first variation formula for the length functional deducing the geodesic equations (\cite[Section 7]{2019arXiv190204015C}), whereas for singular curves the situation is more complicated.
 
The deformability problem of a higher dimensional  immersion $\Phi:\bar{M} \to N$ has been first studied in \cite{2019arXiv190505131C}. The admissibility system of  first order linear PDEs expressing this condition in coordinates  is not easy to study. 
Nonetheless, \cite[Proposition 5.5]{2019arXiv190505131C} shows that only the transversal part $V^{\perp}$ of the vector field $V=V^{\top}+V^{\perp}$ affects the admissibility system. Therefore, in the present work we  consider an adapted tangent basis $E_1,\ldots,E_m$ for the flag \eqref{eq:intanflag} and then we add transversal vector fields $V_{m+1},\ldots,V_n$ of increasing degrees so that a sorting of $\{E_1,\ldots,E_m,V_{m+1},\ldots,V_n\}$ is a local adapted basis for $N$. Then we consider the metric $g$ that makes $E_1,\ldots,E_m,V_{m+1},\ldots,V_n$ an orthonormal basis. Hence we obtain that the admissibility system is equivalent to 
\begin{equation}
\label{eq:intoadmsystem}
E_j( f_{i})=-\sum_{r=m+k+1}^n b_{ i j r} \, f_r - \sum_{h=m+1}^{m+k} a_{ i j h } \, g_h ,
\end{equation}
for $i=m+k+1,\ldots,n$ and $\deg(V_i)>\deg(E_j)$. In equation \eqref{eq:intoadmsystem}  the integer $k$, defined in \eqref{eq:k}, separates the  horizontal control of the systems $ V_{\mh}= \sum_{h=m+1}^{m+k} g_{h} V_h$ from the vertical component $ V_{\mathcal{V}}= \sum_{r=m+k+1}^{n} f_r V_r$.

The presence of isolated submanifolds  and a mild deformability theorem under the strong regularity assumption are showed in \cite{2019arXiv190505131C}. However, the definition of \emph{singularity} for immersed submanifolds, analogous  to the one provided by \cite{MR1189496} in the case of curves, is missing.
Therefore the natural questions that arise are:
\begin{itemize}
\item is  it possible to define a generalization of the \emph{holonomy map} for submanifolds of dimension grater than one?
\item Under what condition does  the surjection of these holonomy map  still imply a deformability theorem in the style of \cite[Theorem 3]{MR1189496}?
\end{itemize}

In the present paper we answer the first question in the cases of ruled $m$-dimensional submanifolds whose $(m-1)$ tangent vector fields $E_2,\ldots,E_m$ have degree $s$ and the first vector field $E_1$ has degree equal to $\iota_0$, where $1\le \iota_0 \le s-1$. The resulting degree is $\deg(M)=(m-1)s+ \iota_0$. Therefore the ruled submanifold is foliated by curves of degree $\iota_0$ out of the characteristic set $\bar{M}_0$, whose points have degree strictly less than $\deg(M)$.
Then, under an exponential
change of coordinates $x=(x_1,\hat{x})$, the admissibility system \eqref{eq:intoadmsystem} becomes 
\begin{equation}
\label{eq:inadmsystem2}
\dfrac{\partial F(x)}{\partial x_1} =-B(x) F(x)-A(x) G(x),
\end{equation}
where $\partial_{x_1}$ is the partial derivative in the direction $E_1$, $G$ are the horizontal coordinates $V_{\mh}= \sum_{h=m+1}^{m+k} g_{h} V_h$, $F$ are the vertical components given by $V_{\mathcal{V}}= \sum_{r=m+k+1}^{n} f_r V_r$ and $A, B$ are matrices defined at the end of Section~\ref{sc:ruledsub}. Therefore, this system of ODEs  is easy to solve in the direction $\partial_{x_1}$ perpendicular to the $(m-1)$ foliation generated by $E_2,\ldots,E_m$. We consider a bounded open set $\Sigma_0 \subset \{x_1=0\}$ in the foliation, then we build the $\eps$-cylinder $\Omega_{\eps}=\{ (x_1,\hat{x}) \ : \ \hat{x} \in \Sigma_0, 0< x_1 <\eps \}$ over $\Sigma_0$. We consider the horizontal controls $G$ in the space of continuous functions compactly supported in $\Omega_{\eps}$. For each fixed $G$ , $F$ is the solution of \eqref{eq:inadmsystem2} vanishing on $\Sigma_0$. Then we can define a higher dimensional \emph{holonomy map} $H_{M}^{\eps}$, whose image is the solution  $F$, evaluated on the top of the cylinder $\Omega_{\eps}$. We say that a ruled submanifold is \emph{regular} when by varying the controls $G$ the image of the holonomy map is a dense subspace, that contains a Schaulder basis  of the Banach space of  continuous vertical functions on $\Sigma_{\eps}$ vanishing at infinity. This Banach space is the closure with respect to the supremum norm of the space of compactly supported vertical functions on $\Sigma_{\eps}$. Namely an immersion is regular if we are able to generate all possible  continuous vertical functions vanishing at infinity on $\Sigma_{\eps}\subset \{x_1=\eps\}$ by letting vary the control $G$ in the space of  continuous horizontal functions vanishing at infinity inside the cylinder $\Omega_{\eps}$. The main difference with respect to the one dimensional case is that the target space of the holonomy map is now the Banach space  of  continuous vertical vector vanishing at infinity on the foliation, instead of the finite vertical space of vectors at the final point $\ga(b)$ of the curve. In Theorem~\ref{th:singchar} we provide a nice characterization of singular ruled submaifolds  in analogy with \cite[Theorem 6]{MR1189496}. 

For general submanifolds there are several obstacles to the construction of a satisfactory generalization of the holonomy map. The main difficulty is that we do not know how to verify a priori the compatibility conditions \cite[Eq. (1.4), Chapter VI]{Hartman}, that are necessary and sufficient conditions  for the uniqueness and the existence of a solution of  the admissibility system \eqref{eq:intoadmsystem} (see \cite[Theorem 3.2, Chapter VI]{Hartman}). In Example~\ref{horizontal submanifolds} we show how we can deal with these compatibility conditions in the particular case of horizontal immersions in the Heisenberg group.

In order to give a positive answer to the second question, we need to consider two additional assumptions on the ruled submanifold: the first one $(i)$ is that the vector fields $E_2,\ldots,E_m$ of degree $s$ \emph{fill  the grading $\mh^1\subset \ldots \subset \mh^s$ from the top}, namely $\dim(\mh^s)- \dim(\mh^{s-1})=m-1$,  and the second one $(ii)$ is that the ruled immersion  foliated by curves of degree $\iota_0$ verifies the bound $s-3 \le \iota_0\le s-1$.  Under these hypotheses  the space of $m$-vector fields of degree grater than $\deg(M)$ is reasonably simple, thus  in Theorem \ref{thm:intcriterion}  we show that each admissible vector field on a regular immersed ruled submanifold is integrable in the spirit of \cite[Theorem 3]{MR1189496}. This result is sharper than the one obtained  for general submanifolds \cite[Theorem~7.3]{2019arXiv190505131C}, where the authors only provide variations of the original immersion compactly supported in an open neighborhood of the strongly regular point. Indeed, since  we solve a differential linear system of equations along the characteristics curves of degree $\iota_0$, we obtain a global result. On the other hand in \cite[Theorem~7.3]{2019arXiv190505131C} the  admissibility system is solved  algebraically assuming a pointwise full rank condition of the matrix $A(\bar{p})$.  To integrate the vector field $V(\bar{p})$ on $\Omega_{\eps}$ we follow the exponential map generating the non-admissible compactly supported variation $\Gamma_{\tau}(\bar{p})=\exp_{\Phi(\bar{p})}(\tau V(\bar{p}))$ of the initial immersion  $\Phi$, where   $\text{supp}(V)\subset \Omega_{\eps}$. By the Implicit Function Theorem  there exists a vector field $Y(\bar{p},\tau)$ on $\Omega_{\eps}$ vanishing on $\Sigma_0$ such that the perturbations  $\tilde{\Gamma}_{\tau}(\bar{p})=\exp_{\Phi(\bar{p})}(\tau V(\bar{p})+ Y(\tau, \bar{p}))$ of $\Gamma$ are immersions of the same degree of $\Phi$ for each $\tau$ small enough. In general $\tilde{\Gamma}$ does not move points on $\Sigma_0$ but changes the values of $\Phi$ on $\Sigma_{\eps}$. Finally, the regularity condition on $\Phi$ allows us to produce the admissible variation that fixes the values on $\Sigma_{\eps}$ and integrate $V$. On the other hand, when the bundle of $m$-vector fields of degree greater than $\deg(M)$ for a general ruled submanifold is larger than the  target space of the higher dimensional holonomy, we lose the surjection in Implicit Function Theorem that allows us to perturb the exponential map to integrate $V$. 

A direct consequence of this result is that the regular ruled immersions of degree $d$ that satisfy the assumption $(i)$ and $(ii)$ are accumulation points for the domain of degree $d$ area functional $A_d(\cdot)$. Therefore it makes sense to consider the first variation formula computed in \cite[Section 8]{2019arXiv190505131C}. An interesting strand of research is  deducing the mean curvature equations for the critical points of the area functional taking into account the restriction embodied by the holonomy map. Contrary to what can be expected, we exhibit in Example~\ref{ex:regularplaneinEngel} a plane foliated by abnormal geodesics of degree one  that is regular and is a critical point for the area functional (since its mean curvature equation vanishes). 

Furthermore these ruled surfaces appear in the study of  the geometrical structures of the visual brain, built by the connectivity between neural cells \cite{cittilibro}. A geometric characterization of the response of  the primary visual cortex in the presence of a visual stimulus from the retina was first described by the D. H. Hubel and T. Wiesel \cite{HW62}, that discovered that the cortical neurons are sensitive to different features such as orientation, curvature, velocity and scale. The so-called simple cells in particular are sensitive to orientation, thus G. Citti and A. Sarti in \cite{CS06} proposed a model where the original image on the retina is lifted to a 2 dimensional surface of maximum degree into the three-dimensional sub-Riemannian manifold  $SE(2)$, adding orientation. In \cite{CS15}  they shows how  minimal surfaces play an important role in the completion process of images. Adding curvature to the model,  a four dimensional Engel structure arises,  see \S~1.5.1.4 in \cite{Petitot2014} and \cite{DobbinsZucker}. When in Example~\ref{ex:visualc} we lift the previous $2D$ surfaces in this structure we obtain surfaces of codimension 2, but their degree is not maximum since we need to take into account the constraint that curvature is the derivative of orientation. Nevertheless these surfaces are ruled, regular and  verify the assumption $(i)$ and $(ii)$, therefore by Theorem~\ref{thm:intcriterion} they can be deformed. Hence, there exists a notion of mean curvature associated to  these ruled surfaces and we might ask if the completion process of images improved for $SE(2)$ based on minimal surfaces can be generalized to this framework. Moreover, if we lift the original retinal image to higher dimensional spaces adding variables that encode new possible features, as suggested in \cite{NCS18} following even  a non-differential approach based on metric spaces, we may ask if the lifted surfaces  are still ruled and regular.

The paper is organized as follows. In Section~\ref{sc:prel} we recall the definitions of graded manifolds, degree of a submanifold, admissible variations and admissible vector fields. In Section~\ref{sc:admsys} we deduce the admissibility system \eqref{eq:intoadmsystem}. In Section~\ref{sc:ruledsub} we provide the definition of ruled submanifolds. 
Section~\ref{sc:holonomy} is completely devoted to the description of the higher-dimensional holonomy map and characterization of regular and singular ruled submanifolds. Finally, in Section~\ref{sc:intheo} we give the proof of Theorem \ref{thm:intcriterion}.

\subsubsection*{Acknowledgement}
I warmly thank my Ph.D. supervisors Giovanna Citti and  Manuel Ritoré for their advice and for fruitful discussions that gave rise to the idea of higher dimensional holonomy map. I would also like to thank Noemi Montobbio for an interesting  conversation on proper subspaces of Banach spaces and the referee for her/his useful comments.
\section{Preliminaries}
\label{sc:prel}
Let $N$ be an $n$-dimensional smooth manifold. Given two smooth vector fields $X,Y$ on $N$, their \emph{commutator} or \emph{Lie bracket} is defined by $[X,Y]:=XY-YX$.  An \emph{increasing filtration} $(\mh^i)_{i\in \nn}$ of  the tangent bundle $TN$ is a flag of sub-bundles
\begin{equation}
\mathcal{H}^1\subset\mathcal{H}^2\subset\cdots\subset \mathcal{H}^i\subset\cdots\subseteq TN,
\label{manifold flag}
\end{equation}
such that
\begin{enumerate}
\label{def:incfilt}
\item[(i)] $ \cup_{i \in \nn} \mh^i= TN$ 
\item[(ii)]
 $ [\mathcal{H}^{i},\mathcal{H}^{j}] \subseteq \mathcal{H}^{i+j},$ for $ i,j \ge1$, 
\end{enumerate}
where  
$ [\mathcal{H}^i,\mathcal{H}^j]:=\{[X,Y]  : X \in \mathcal{H}^i,Y \in \mathcal{H}^j\}$.
Moreover, we say that an increasing filtration is \emph{locally finite} when
\begin{enumerate}
\item[(iii)]  for each $p \in N$ there exists an integer  $s=s(p)$ satisfying $\mathcal{H}^s_p=T_p N$. The \emph{step} at $p$ is the least integer $s$ that satisfies the previous property. Then  we have the following flag of subspaces
\begin{equation}
 \mathcal{H}^1_p\subset\mathcal{H}^2_p\subset\cdots\subset \mathcal{H}^s_p=T_p N.
 \label{flag at each point}
\end{equation}
\end{enumerate}

A \textit{graded manifold} $(N,(\mh^i))$ is a smooth manifold $N$ endowed
with a locally finite increasing filtration, namely  a flag of sub-bundles \eqref{manifold flag} satisfying (i),(ii) and (iii). For the sake of brevity a locally finite increasing filtration will be simply called a filtration.
Setting $n_i(p):=\dim {\mathcal{H}}^i_p $, the integer list $(n_1(p),\cdots,n_s(p))$  is called the \textit{growth vector} of the filtration \eqref{manifold flag} at $p$. When the growth vector is constant in a neighborhood of a point $p \in N$ we say that $p$ is a \textit{regular point} for the filtration. We say that a filtration $(\mathcal{H}^i)$ on a manifold $N$ is \textit{equiregular} if the growth vector is constant in $N$. From now on we suppose that $N$ is an equiregular graded manifold.

Given a vector $v$ in $T_p N$ we say that the \textit{degree} of $v$ is equal to $\ell$ if $v\in\mathcal{H}_p^\ell$ and $v \notin \mathcal{H}_p^{\ell-1}$. In this case we write $\text{deg}(v)=\ell$. The degree of a vector field is defined pointwise and can take different values at different points.

Let $(N,(\mh^1,\ldots, \mh^s))$ be an equiregular graded manifold. Take $p\in N$ and consider an open neighborhood $U$ of $p$  where a local frame $\{X_{1},\cdots,X_{n_1}\}$  generating  $\mathcal{H}^1$ is defined. Clearly the degree of $X_j$, for $j=1,\ldots,n_1$, is equal to one since the vector fields $X_1,\ldots,X_{n_1}$ belong to $\mh^1$. Moreover the vector fields $X_1, \ldots,X_{n_1}$ also lie in $\mh^2$, we add some vector fields $X_{n_{1}+1},\cdots,X_{n_2} \in \mathcal{H}^2\setminus \mathcal{H}^{1} $ so that $(X_1)_p,\ldots,(X_{n_2})_p$ generate $\mathcal{H}^2_p$. Reducing $U$ if necessary we have that $X_1,\ldots,X_{n_2}$ generate $\mathcal{H}^2$ in $U$.  Iterating this procedure we obtain a basis of $TM$ in a neighborhood of $p$
\begin{equation}
\label{local adapted basis to the bundle}
 (X_1,\ldots,X_{n_1},X_{n_1+1},\ldots,X_{n_2},\ldots,X_{n_{s-1}+1}, \ldots,X_n),
\end{equation}
such that the vector fields $X_{n_{i-1}+1},\ldots,X_{n_i}$ have degree equal to $i$, where $n_0:=0$. The basis obtained in (\ref{local adapted basis to the bundle}) is called an \textit{adapted basis} to the filtration $(\mh^1,\ldots,\mh^s)$.
Given an adapted basis $(X_i)_{1\le i\le n}$, the \emph{degree} of the \emph{simple} $m$-vector field $X_{j_1}\wedge\ldots\wedge X_{j_m}$ is defined by
\[
\deg(X_{j_1}\wedge\ldots\wedge X_{j_m}):=\sum_{i=1}^m\deg(X_{j_i}).
\]
Any $m$-vector $X$ can be expressed as a sum
\[
X_p=\sum_{J}\lambda_J(p)(X_J)_p,
\]
where $J=(j_1,\ldots,j_m)$, $1\le j_1<\cdots <j_m\le n$, is an ordered multi-index, and $X_J:=X_{j_1}\wedge\ldots\wedge X_{j_m}$. The degree of $X$ at $p$ with respect to the adapted basis $(X_i)_{1\le i\le n}$ is defined by
\[
\max\{\deg((X_J)_p):\la_J(p)\neq 0\}.
\]
It can be easily checked that the degree of $X$ is independent of the choice of the adapted basis and it is denoted by $\deg(X)$.

If $X=\sum_J\la_J X_J$ is an $m$-vector expressed as a linear combination of simple $m$-vectors $X_J$, its projection onto the subset of $m$-vectors of degree $d$ is given by
\begin{equation}
\label{projection onto degree d}
(X)_d=\sum_{\deg(X_J)=d} \la_JX_J,
\end{equation}
and its projection over the subset of $m$-vectors of degree larger than $d$ by
\[
\Pi_d(X)=\sum_{\deg(X_J)\ge d+1} \la_JX_J.
\]

In an equiregular graded manifold with a local adapted basis $(X_1, \ldots,X_n)$, defined as in (\ref{local adapted basis to the bundle}), the maximal degree that can be achieved by an $m$-vector, $m\le n$,  is the integer $d_{\max}^m$ defined by 
\begin{equation}
\label{maximum degree over N}
 d_{\max}^m:=\deg(X_{n-m+1})+\cdots+\deg(X_{n}).
\end{equation}

\subsection{Degree of a submanifold}
Let $\Phi:\bar{M} \to N$ be a $C^1$ immersion in an equiregular graded manifold $(N,(\mh^1,\ldots, \mh^s))$ such that $\dim(\bar{M})=m<n=\dim(N)$. Following \cite{DonneMagnani, vittonemagnani}, we define the degree of $M=\Phi(\bar{M})$ at a point $\bar{p}\in \bar{M}$ by 
\[
\deg_M(\bar{p}):=\deg(v_1\wedge\ldots\wedge v_m),
\]
where $v_1,\ldots,v_m$ is a basis of $T_{\bar{p}}M=(d\Phi)_{\bar{p}}(T_{\bar{p}} \bar{M})$ and $d\Phi$. We  denote by $T_{\bar{p}} M=(d\Phi)_{\bar{p}}(T_{\bar{p}} \bar{M})$ the tangent space at $p=\Phi(\bar{p})$, where $(d\Phi)_{\bar{p}}$ is the differential of $\Phi$ at $\bar{p} \in \bar{M}$. We use this notation in order to emphasize that we consider the tangent space of the image $\Phi(\bar{p})$ of a fixed point $\bar{p}$ in $\bar{M}$. The \textit{degree} $\deg(M)$ of a immersed submanifold $M$ is the integer 
\[
\deg(M):=\max_{\bar{p}\in \bar{M}} \deg_{M}(\bar{p}).
\]
We define the \textit{singular set} of a submanifold $M$ by  $\Phi(\bar{M}_0)$ where
\begin{equation}
\label{singular set}
 \bar{M}_0=\{\bar{p} \in \bar{M} : \deg_M(\bar{p})<\deg(M) \}.
\end{equation}
Singular points can have different degrees between $m$ and $\deg(M)-1$.  Following \cite[0.6.B]{Gromov} an alternative way to define the pointwise degree is by means of the formula  
\begin{equation*}
 \deg_M(\bar{p})= \sum_{j=1}^s j (\tilde{m}_j(\bar{p})- \tilde{m}_{j-1} (\bar{p})),
\end{equation*}
setting $\tilde{m}_{0}=0$ and $\tilde{m}_j (\bar{p})=\text{dim}(T_{\bar{p}}M \cap \mathcal{H}_p^j )$. Namely, the degree is the homogenous dimension of the flag
\begin{equation}
\label{eq:tanflag}
 \tilde{\mathcal{H}}_p^1 \subset \tilde{\mathcal{H}}_p^2 \subset\cdots\subset \tilde{\mathcal{H}}_p^s =T_{\bar{p}}M,
\end{equation}
where $\tilde{\mh}_p^j:=T_{\bar{p}}M \cap \mathcal{H}_p^j$. As we pointed out in \cite[Section 3]{2019arXiv190505131C} the area functional associated to an immersed sumbanifold depends on the degree.

\begin{definition}
\label{def:areafunctional}
Let $M$ be a $C^1$ immersed submanifold of degree $d=\deg(M)$ in an  equiregular graded manifold $(N,\mathcal{H}^1,\ldots, \mh^s)$ endowed with a Riemannian metric $g$. Let $\mu$ be a Riemannian metric in $M$ and $e_1,\ldots,e_m$ be  a $\mu$ orthonormal basis. Then the \emph{degree $d$ area} $A_d$ is defined by
 \begin{equation}
 \label{eq:projected_integral_formula_Ad}
  A_d(M')=\int_{{M'}} 
 |\left( e_1\wedge \ldots \wedge e_m \right)_d|_{g} \ d\mu(p),
 \end{equation}
for any bounded measurable set $M'\subset M$ and where $d\mu$ is the  Riemannian volume given by $\mu$. In the previous formula $(\cdot)_d$ denotes the projection onto the subset of $m$-vectors of degree $d$ defined in \eqref{projection onto degree d}.
\end{definition}

\subsection{Admissible variations and admissibility system of PDEs}
Given a graded manifold $(N,\mh^1,\ldots,\mh^s)$, we consider a generic Riemannian metric $g=\escpr{\cdot, \cdot}$ on $TN$.  Let $\Phi:\bar{M} \to N$ be a smooth
immersion in $N$, we set $M=\Phi(\bar{M})$ and $d=\deg(M)$. Let $(X_i)_i$ be a local adapted basis around $p \in M$.
Following \cite[Section 5]{2019arXiv190505131C} we recall the notions of admissible variation, its variational vector field, admissible and integrable vector field.
\begin{definition}
\label{def:admissible}
A smooth map $\Gamma:\bar{M}\times (-\eps,\eps)\to N$ is said to be \emph{an admissible variation} of $\Phi$ if $\Gamma_t:\bar{M}\to N$, defined by $\Gamma_t(\bar{q}):=\Gamma(\bar{q},t)$, satisfies the following properties
\begin{enumerate}
\item[(i)] $\Gamma_0=\Phi$,
\item[(ii)] $\Gamma_t(\bar{M})$  is an immersion of the same degree as $\Phi(\bar{M})$ for small enough $t$, and
\item[(iii)] $\Gamma_t(\bar{q})=\Phi(\bar{q})$ for $\bar{q}$ outside of a given compact set of $\bar{M}$.
\end{enumerate}
\end{definition}

\begin{definition}
Given an admissible variation $\Gamma$, the \emph{associated variational vector field} is defined by
\begin{equation}
\label{eq:admissibleV}
V(\bar{q}):=\frac{\ptl\Gamma}{\ptl t}(\bar{q},0).
\end{equation}
\end{definition}

Let $\mathfrak{X}_0(\bar{M},N)$ be the space of compactly supported smooth vector fields on $\bar{M}$ with value in $N$.  Since it turns out that variational vector fields associated to an admissible variations satisfy the system \eqref{eq:system of PDEs for admissible} (see \cite[Section 5]{2019arXiv190505131C}) we are led to the following definition

\begin{definition}
\label{def:adm}
Given an immersion $\Phi:\bar{M}\to N$, a vector field $V\in \mathfrak{X}_0(\bar{M},N)$ is said to be \emph{admissible}  if it satisfies the system of first order PDEs
 \begin{equation}
 \label{eq:system of PDEs for admissible}
 0=\langle e_1\wedge\ldots \wedge e_m,\nabla_{V(\bar{q})}X_J\rangle +\sum_{j=1}^m \langle e_1 \wedge \ldots \wedge\nabla_{e_j}  V \wedge \ldots \wedge e_m ,X_J\rangle 
\end{equation}
where $X_J=X_{j_1} \wedge \ldots \wedge X_{j_m}$,  $\deg(X_J)>d$ and $e_1,\ldots,e_m$ is basis of $T_{\bar{q}} M$, for each $\bar{q}$ in $\bar{M}$ such that $q= \Phi(\bar{q})$.
\end{definition}
\begin{definition}
We say that an admissible vector field $V\in\mathfrak{X}_0(\bar{M},N)$ is \emph{integrable} if there exists an admissible variation such that the associated variational vector field is $V$.
\end{definition}

\section{Intrinsic coordinates for the admissibility system of PDEs}
\label{sc:admsys}
Let $\Phi:\bar{M} \to N$ be a smooth immersion in a graded manifold, $M=\Phi(\bar{M})$ and $d=\deg(M)$. By \cite[Proposition 6.4]{2019arXiv190505131C} we realize that the admissibility of a vector field $V$ is independent of the metric. Therefore we can use any metric in order to study the system. Let $\bar{p}$ be a point in $\bar{M}$ such that $p=\Phi(\bar{p})$ is a point in $M \setminus M_0$, that is an open set thanks to \cite[Corollary 2.4]{2019arXiv190505131C}. Then there exists  an open neighborhood $\bar{O}' \subset \bar{M}$ of $\bar{p}$ such that $\Phi(\bar{O}')$ has fixed degree $d$. Moreover, we can always find  an open neighborhood $\bar{O} \subset \bar{O}'$ such that $\Phi(\bar{O})=O$ is an embedding of fixed degree $d$ . From now on we will consider this piece of submanifold $O$.

Letting $TO$ be the tangent bundle of $O$, we consider the subbundle $\tilde{\mh}^i=TO \cap \mh^{i}$ for each $i=1,\ldots,s$. Then the submanifold $O$ inherits from the ambient space an increasing filtration $\tilde{\mh}^1 \subset \ldots \subset \tilde{\mh}^s$, pointwise given by the flag \eqref{eq:tanflag}, that makes $(O, \tilde{\mh}^1,\ldots, \tilde{\mh}^s)$ a graded structure. Evidently, (i) in Definition~\ref{def:incfilt} is satisfied. On the other hand, if $X\in\tilde{\mathcal{H}}^i$ and $Y\in\tilde{\mh}^j$, we can extend both vector fields in a neighborhood of $N$ so that the extensions $X_1$, $Y_1$ lie in $\mh^i$ and $\mh^j$, respectively. Then $[X,Y]$ is a tangent vector to $O$ that coincides on $N$ with $[X_1,Y_1]\in \mh^{i+j}$. Hence $[X,Y]\in\tilde{\mh}^{i+j}$. This implies condition (ii) in Definition~\ref{def:incfilt}. Moreover, $(O, \tilde{\mh}^1,\ldots, \tilde{\mh}^s)$ is also equiregular by \cite[Proposition 3.7]{2019arXiv190505131C}, since the degree is constant equal to $d$ on $O$. Reducing $O$ if necessary, following the same argument of Section \ref{sc:prel}, there exists a local adapted basis $(\tilde{E}_1,\ldots,\tilde{E}_m)$ to the filtration $\tilde{\mh}^1\subset \ldots \subset \tilde{\mh}^s$. For each $j=1,\ldots,m$ we set $\deg(\tilde{E}_j)=\ell_j$, then we can  extend each vector field $\tilde{E}_j$ in a neighborhood $U$ of $N$ around $p$ so that the extensions $E_j$ lie in $\mh^{\ell_j}$. 
 Finally we complete this basis of vector fields $(E_1,\ldots,E_m)$ to a basis of the ambient space $T U$ adding the vector fields $V_{m+1},…,V_n$ of increasing degree such that a sorting of $\{E_1,…,E_m, V_{m+1},…,V_n\}$ is an adapted basis of  $T U$. Then we consider the metric $g=\escpr{\cdot,\cdot}$ that makes  $E_1,\ldots,E_m,V_{m+1},\ldots,V_n$ an orthonormal basis in a neighborhood $U$ of $p$.  We will denote by $(Y_1,\ldots,Y_n)$ the local adapted basis generated by this sorting of $E_1,\ldots,E_m,V_{m+1},\ldots,V_n$. From now on we will denote also denote $(\tilde{E}_1,\ldots,\tilde{E}_m)$ by $(E_1,\ldots,E_m)$ with a little abuse of notation.
\begin{definition}
Letting $\iota_0$ be the integer defined by 
\begin{equation}
\label{eq:iota}
\iota_0(O)=\max_{\bar{p} \in \bar{O}}  \min_{1\le \ell \le s} \{ \alpha  :  \tilde{m}_{\ell} (\bar{p}) \ne 0 \},
\end{equation}
 we set 
\begin{equation}
\label{eq:k}
k:=n_{\iota_0}- \tilde{m}_{\iota_0},
\end{equation}
where $ \tilde{m}_{\ell} (\bar{p})$ is defined a line before of equation \eqref{eq:tanflag}.
\end{definition}
\begin{remark}
\label{rk:escprW}
Let $W$ be a vector field on $U \subset N$. Having in mind the equation \eqref{eq:system of PDEs for admissible} we consider the scalar product 
\begin{equation}
\label{eq:scpre}
\escpr{E_1\wdw \overset{(j)} {W} \wdw E_m, Y_{\ell_1} \wdw Y_{\ell_m}}
\end{equation}
where $J=(\ell_1,\ldots,\ell_m)$, $1\le \ell_1\le \ldots \le \ell_m \le n$ and 
\[
\deg(Y_{\ell_1})+ \ldots+ \deg(Y_{\ell_m})\ge d+1.
\]
If there are at least two $Y_{\ell_\alpha},Y_{\ell_\beta} \in \{V_{m+1},\ldots,V_{n} \}$ with  $\alpha,\beta \in \{1,\ldots,m\}$ then \eqref{eq:scpre} is equal to zero thanks to the orthogonal assumption of the basis $E_1,\ldots,E_m,$ $V_{m+1},\ldots,V_n$. 
If $\{Y_{\ell_1},\ldots, Y_{\ell_m}\} \cap\{V_{m+1},\ldots,V_n\}= \emptyset$, then $\deg(Y_{\ell_1} \wdw Y_{\ell_m})\le d$. 
Finally, if there exists only one $Y_{\ell_\alpha}=V_i$ for some $i \in \{m+1,\ldots,n\}$ then $\alpha=j$, if not  \eqref{eq:scpre} is equal to zero by orthogonality assumption of the basis $E_1,\ldots,E_m$. Then, denoting by $\sigma_{i}^{j}$ the permutation caused by the reordering and by   $\text{sgn}(\sigma_{i}^{j})=\pm 1$ its sign, we have 
\[
Y_J=\text{sgn}(\sigma_{i}^{j}) E_1 \wdw \overset{(j)} {V_i} \wdw E_m.
\]
Since $$\deg(E_1 \wdw \overset{(j)} {V_i} \wdw E_m )> d$$ we deduce that $\deg(V_i) > \deg(E_j)$. Then \eqref{eq:scpre} coincides with 
\begin{align*}
&\text{sgn}(\sigma_{i}^{j}) \escpr{E_1\wdw \overset{(j)} {W} \wdw E_m, E_1 \wdw \overset{(j)} {V_i} \wdw E_m }\\
&=\text{sgn}(\sigma_{i}^{j}) \sum_{l=m+1}^n \escpr{W,V_l} \text{sgn}(\sigma_{l}^{j}) \delta_{l,i}\\
&=\text{sgn}(\sigma_{i}^{j})^2 \escpr{W,V_i}=\escpr{W,V_i},
\end{align*}
where $\delta_{l,i}$ is the Kronecker delta.
Since $(V_i)_i$ and $(E_j)_j$ have increasing degree, we obtain  $\deg(V_i) > \deg(E_1)=\iota_0$ if and only if $i=m+k+1,\ldots,n$, where $k$ is defined in \eqref{eq:k}. 
Therefore  we deduce that the only $m$-vectors $Y_{\ell_1} \wdw Y_{\ell_m}$ of degree strictly greater than $d$ such that \eqref{eq:scpre}  is different from zero are 
\[
\text{sgn}(\sigma_{i}^{j}) E_1 \wdw \overset{(j)} {V_i} \wdw E_m,
\]
for $i=m+k+1,\ldots,n$ and $\deg(V_i)>\deg(E_j)$. 
\end{remark}

\begin{definition}
We say that a vector field $V_l  \in \{V_{m+1},\ldots,V_{n}\}$ is horizontal if $ \deg(V_l) \le \deg(E_1)=\iota_0$ and 
is vertical if $\deg(V_l) > \deg(E_1)=\iota_0$. The horizontal bundle $\mh$ is generated by $V_{m+1},\ldots,V_{m+k}$ and the vertical bundle $\mathcal{V}$ is generated by $V_{m+k+1},\ldots,V_{n}$, where $k=n_{\iota_0}-\tilde{m}_{\iota_0}$. 
\end{definition}

Thanks to \cite[Proposition  5.5]{2019arXiv190505131C} we know that $V \in \mathfrak{X}_0(\bar{M},N)$ is admissible if and only if 
\begin{equation}
\label{eq:perpvector}
V^{\perp}= \sum_{h=m+1}^{m+k} g_h V_h+ \sum_{r=m+k+1}^{n} f_r  V_r
\end{equation}
is admissible. We denote by $V^{\perp}_{\mh}=\sum_{h=m+1}^{m+k} g_h V_h$ (resp. $V^{\perp}_{\mathcal{V}}=\sum_{r=m+k+1}^{n} f_r  V_r$) the horizontal projection on $\mh$ (resp. the vertical projection on $\mathcal{V}$). For $h=m+1,\ldots,m+k$ and $r=m+k+1,\ldots,n$, $g_h,f_r$ are smooth functions on $O$  and when we evaluate the vector field $V^{\perp}$ at $\bar{q} \in \bar{O}$ we mean 
\[
V^{\perp}(\bar{q})=\sum_{h=m+1}^{m+k} g_h(\Phi(\bar{q})) (V_h)_{\Phi(\bar{q})}+ \sum_{r=m+k+1}^{n} f_r (\Phi(\bar{q}))  (V_r)_{\Phi(\bar{q})}.
\]
Therefore, locally we can consider the vector field $V^{\perp}$ defined on $O$ and extend $V^{\perp}$ to the open neighborhoood $U \subset N$. Then, putting $V^{\perp}$ in \eqref{eq:system of PDEs for admissible} we have 
\begin{equation}
\label{eq:systemPDEsLiebrackets}
\begin{aligned}
0=&\escpr{E_1\wdw E_m,\nabla_{V^{\perp}} Y_J}+\sum_{j=1}^m \langle E_1 \wedge \ldots \wedge\nabla_{E_j}  V^{\perp} \wedge \ldots \wedge E_m , Y_J\rangle \\
=&\sum_{j=1}^m - \escpr{E_1\wdw \nabla_{V^{\perp}} E_j \wdw E_m, Y_J}\\
&+ \langle E_1 \wedge \ldots \wedge\nabla_{E_j}  V^{\perp} \wedge \ldots \wedge E_m , Y_J\rangle\\
=&\sum_{j=1}^m \langle E_1 \wedge \ldots \wedge \overset{(j)}{(\nabla_{E_j}  V^{\perp}-\nabla_{V^{\perp}} E_j) } \wedge \ldots \wedge E_m , Y_J\rangle \\
=& \sum_{j=1}^m \langle E_1 \wedge \ldots \wedge \overset{(j)}{[E_j,V^{\perp}] } \wedge \ldots \wedge E_m , Y_J\rangle.
\end{aligned}
\end{equation}
By Remark~\ref{rk:escprW} we have to consider the scalar product only with the $m$-vector $$Y_J=\text{sgn}(\sigma_{i}^{\alpha}) E_1 \wdw  \overset{(\alpha)} {V_i}  \wdw E_m$$ for $i=m+k+1,\ldots,n$, $\alpha=1,\ldots,m$, $\deg(V_i)>\deg(E_{\alpha})$ and $\text{sgn}(\sigma_{i}^{\alpha})=\pm 1$ is the sign of the permutation $\sigma_{i}^{\alpha}$ caused by the reordering. By substituting the expression \eqref{eq:perpvector} of  $V^{\perp}$ in equation \eqref{eq:systemPDEsLiebrackets}, we obtain that \eqref{eq:system of PDEs for admissible} is equivalent to
\begin{equation}
\label{eq:local system of PDEs2}
\begin{aligned}
  \sum_{j=1}^m &\Big( \sum_{r=m+k+1}^n \tilde{c}_{i j r \alpha}  E_j (f_r)+ \sum_{h=m+1}^{m+k} \tilde{c}_{i j h \alpha}  E_j (g_h)   \\
  &+  \sum_{r=m+k+1}^n \tilde{b}_{ i j r \alpha} f_r + \sum_{h=m+1}^{m+k} \tilde{a}_{i j h \alpha} g_h  \Big) =0,
\end{aligned}
\end{equation}
where
\begin{align*}
\tilde{c}_{i j t \alpha}&= \text{sgn}(\sigma_{t}^{j})  \text{sgn}(\sigma_{i}^{\alpha}) \escpr{E_1 \wdw \overset{(j)} {V_t} \wdw E_m ,E_1 \wdw \overset{(\alpha)} {V_i} \wdw E_m }\\
\tilde{a}_{i j h \alpha }&= \text{sgn}(\sigma_{i}^{\alpha}) \escpr{ E_1 \wedge \ldots \wedge \overset{(j)} {[E_j, V_h] } \wedge \ldots \wedge E_m, E_1 \wdw  \overset{(\alpha)} {V_i}  \wdw E_m } \\
\tilde{b}_{i j r \alpha }&= \text{sgn}(\sigma_{i}^{\alpha}) \escpr{ E_1 \wedge \ldots \wedge \overset{(j)} {[E_j, V_r]} \wedge \ldots \wedge E_m, E_1 \wdw  \overset{(\alpha)} {V_i}  \wdw E_m },\\
\end{align*}
for  $t=m+1,\ldots,n$, $r=m+k+1,\ldots,n$, $h=m+1,\ldots,m+k$, $\alpha=1,\ldots,m$, $i=m+k+1,\ldots,n$ and $\deg(V_i)>\deg(E_{\alpha})$. Then we have that  $\tilde{c}_{i j t \alpha}$
is equal to $1$ for $i=t>m+k$, $\alpha=j$ and $\deg(V_i)>\deg(E_j)$  or equal to zero otherwise. Moreover, we notice  by Remark~\ref{rk:escprW} that $\tilde{a}_{i j h \alpha }$ and $\tilde{b}_{i j r \alpha }$ are different from zero only when $\alpha=j$ and in particular  we have
\begin{equation}
\label{eq:newA}
a_{i j h}:=\tilde{a}_{i j h j}=\escpr{ V_i, [E_j,V_h]},
\end{equation}
for $h=m+1, \ldots,m+k$, $i=m+k+1,\ldots,n$, $\deg(V_i)>\deg(E_j)$ and 
\begin{equation}
\label{eq:newB}
b_{i j r}:=\tilde{b}_{i j r j }= \escpr{V_i, [E_j,V_r]},
\end{equation}
for $i,r=m+k+1,\ldots,n$ and $\deg(V_i)>\deg(E_j)$. Hence $V$ is admissible if and only if 
\begin{equation}
\label{eq:admsystem}
E_j( f_{i})=-\sum_{r=m+k+1}^n b_{ i j r} \, f_r - \sum_{h=m+1}^{m+k} a_{ i j h } \, g_h ,
\end{equation}
for $i=m+k+1,\ldots,n$ and $\deg(V_i)>\deg(E_j)$.
\begin{remark} \mbox{}
\label{rk:Lip}
\begin{enumerate} 
\item[1.] In all the previous computations we strongly used the tools of differential geometry such as the covariant derivative and the Levi-Civita connection. However, we  notice that the coefficients $a_{i j h}$ and $b_{i j r}$ are defined almost everywhere if we only assume that the vector fields $E_1,\ldots,E_m,V_{m+1},\ldots,V_n$ are Lipschitz continuous. Indeed, under this Lipschitz assumption, the Lie brackets $[E_j,V_h]$ and $[E_j,V_r]$ for $j=1,\ldots,m$, $h=m+1, \ldots,m+k$ and $r=m+k+1,\ldots,n$  are defined almost everywhere, thanks to  \cite{EF17}. Therefore it would be interesting to consider $C^{1,1}$ immersions and deducing the admissibility system \eqref{eq:admsystem} in a weak formulation using the tools of first order differential calculus for general metric measure spaces, developed in recent years by \cite{Cheeger99,HKST15,Gigli18, AGS13}.
\item [2.] Even in this smooth setting we realize that in the admissibility system \eqref{eq:admsystem}  we can consider  the functions  $f_{m+k+1},\ldots,f_n$ to be  continuously differentiable on $O$ and $g_{m+1},\ldots,g_{m+k}$ in the class of continuous functions on $O$.
\end{enumerate}
\end{remark}

\begin{example}[Horizontal submanifolds]
\label{horizontal submanifolds}
Given $n>1$ we consider the Heisenberg group  $\mathbb{H}^n$, defined as $\rr^{2n+1}$ endowed with the distribution $\mh$ generated by 
\[
X_i=\dfrac{\partial}{\partial {x_i}}+\dfrac{y_i}{2}\dfrac{\partial}{\partial t}, \quad
Y_i=\dfrac{\partial}{\partial {y_i}}-\dfrac{x_i}{2}\dfrac{\partial}{\partial t} \quad i=1,\ldots,n. 
\]
 
The Reeb vector fields is provided by $T=\partial_t=[X_i,Y_i]$ for $i=1,\ldots,n$ and has degree equal to $2$. Let $g=\escpr{\cdot,\cdot}$ be the Riemannian metric that makes $(X_1,\ldots,X_n, Y_1,\ldots,Y_n,T)$ an orthonormal basis.  Let $\Omega$ be an open set of $\rr^m$, with $m \le n$. Here we consider a smooth immersion $\Phi:\Omega \to \mathbb{H}^n$ such  $M=\Phi(\Omega)$ is a horizontal submanifold.  Let $E_1,\ldots,E_m$ be an orthonormal local frame, then we have
\begin{equation}
E_{j}=\sum_{i=1}^{n} {\alpha}_{j i}  X_i + \beta_{j i}  Y_i \quad \text{for} \quad j=1,\ldots,m,
\end{equation}
where ${\alpha}_{j i}=\escpr{E_j,X_i}$, ${\beta}_{j i}=\escpr{E_j,Y_i}$, $\sum_{i=1}^n {\alpha}_{j i}^2+ \beta_{j i}^2=1$ for each $j=1,\ldots,m$ and the matrix $$A(\bar{p})=\left((\alpha_{ji})(\Phi(\bar{p}))| (\beta_{ji})(\Phi(\bar{p}))\right)_{j=1,\ldots,m}^{i=1,\ldots,n}$$ has full rank equal to $m$, for each $\bar{p}\in \Omega$. Since $M$ is horizontal we also have that 
\[
\escpr{[E_j,E_{\nu}],T}=\sum_{i=1}^n \sum_{k=1}^n \escpr{[\alpha_{ji}X_i+\beta_{ji}Y_i,\alpha_{\nu k}X_k+\beta_{\nu k}Y_k],T}=0,
\]
that is equivalent  to
\begin{equation}
\label{eq:albeT}
\sum_{i=1}^n \alpha_{ji} \beta_{\nu i}-\beta_{ji} \alpha_{\nu i}=0.
\end{equation}
Therefore a vector field $V=\sum_{l=1}^n g_l \, X_i + g_{l+n} \, Y_l + f \,T$ is admissible if and only if it satisfies the system  \eqref{eq:admsystem}, that in this case is given by
\[
E_j(f)=-\escpr{[E_j,T],T} f - \left(\sum_{h=1}^n ( \escpr{[E_j,X_h],T} g_h +  \escpr{[E_j,Y_l],T} g_{h+n}) \right),
\]
for $j=1,\ldots,m$. A straightforward computation shows that this system is equivalent to 
\begin{equation}
\label{eq:Lagrangiansy}
E_j(f)=\sum_{i=1}^n \beta_{ji} g_{i} - \alpha_{ji} g_{i+n} \quad \text{for} \quad j=1,\ldots,m.
\end{equation}
A necessary and sufficient conditions  for the uniqueness and the existence of a solution of  the admissibility system \eqref{eq:Lagrangiansy} (see \cite[Theorem 3.2, Chapter VI]{Hartman}) are given by
\begin{equation}
\label{eq:compcond}
E_j E_{\nu} (f)- E_{\nu}E_j(f)=E_j \left(\sum_{i=1}^n \beta_{\nu i} g_{i} - \alpha_{\nu i} g_{i+n}  \right)-E_{\nu}\left( \sum_{i=1}^n \beta_{ji} g_{i} - \alpha_{ji} g_{i+n} \right),
\end{equation}
for each $j,\nu=1,\ldots,m$. These are the so called integrability condition \cite[Eq. (1.4), Chapter VI]{Hartman}. A straightforward computation shows that the right hand side of is equal to
\begin{equation}
\begin{aligned}
&\sum_{i=1}^n (E_j(\beta_{\nu i})- E_{\nu}(\beta_{ji})) g_i + (E_{\nu}(\alpha_{ji})-E_j( \alpha_{\nu i})) g_{i+n}\\
&+\sum_{i=1}^n \beta_{\nu i} E_j(g_i) - \beta_{ji} E_{\nu}(g_i) -\alpha_{\nu i} E_j(g_{i+n}) + \alpha_{ji} E_{\nu}(g_{i+n}).
\end{aligned}
\end{equation}
Moreover, the left hand side is equal to 
\begin{align*}
[E_j,E_{\nu}](f)&=\sum_{k=1}^m c_{j \nu}^{k} E_k(f)=\sum_{k=1}^m c_{j \nu}^{k} \sum_{i=1}^n \beta_{ki} g_{i} - \alpha_{ki} g_{i+n} \\
&=\sum_{i=1}^n \escpr{[E_j,E_{\nu}],Y_{i}} g_{i}-  \escpr{[E_j,E_{\nu}],X_{i}} g_{i+n}\\
&=\sum_{i=1}^n (E_j(\beta_{\nu i})-E_{\nu}(\beta_{ji} ) ) g_{i} + (E_{\nu}(\alpha_{ji})-E_j(\alpha_{\nu i}))g_{i+n}.
\end{align*}
Therefore the compatibility (or integrability) conditions are given by 
\begin{equation}
\label{eq:ccLM}
\sum_{i=1}^n \beta_{\nu i} E_j(g_i) - \beta_{ji} E_{\nu}(g_i) -\alpha_{\nu i} E_j(g_{i+n}) + \alpha_{ji} E_{\nu}(g_{i+n})=0,
\end{equation}
for each $\nu,j=1,\ldots,m$. Moreover, taking into account \eqref{eq:albeT}, the equation \eqref{eq:ccLM} is equivalent to 
\begin{equation}
\label{eq:ccLM2}
\begin{aligned}
&\sum_{ i\ne k}^n (\beta_{\nu i} \alpha_{jk}-\beta_{ji} \alpha_{\nu k})X_k(g_i)+ (\beta_{\nu i} \beta_{jk}-\beta_{ji} \beta_{\nu k})Y_k(g_{i})\\
&+\sum_{ i\ne k}^n (\alpha_{j i} \alpha_{\nu k}-\alpha_{\nu i} \alpha_{j k})X_k(g_{i+n})+ (\alpha_{j i} \beta_{\nu k}-\beta_{\nu i} \beta_{j k})Y_k(g_{i+n})=0.
\end{aligned}
\end{equation}
\end{example}

\begin{remark}
Notice that if we want to find a solution $f$ of \eqref{eq:Lagrangiansy}, the controls $g_{i},\ldots,g_{2n}$ have to verify the compatibility conditions \eqref{eq:ccLM2}. Therefore  to obtain a suitable generalization of the holonomy map (defined for curves in  \cite[Section 5]{2019arXiv190204015C}) we need to consider the subspace of the space of horizontal vector fields on $M$ that verify  \eqref{eq:ccLM2}. We recognize that studying the holonomy map for these horizontal immersions is engaging problem that have been investigated by \cite{Gromov86,Pansu16}, but in the present work we will consider different kind of immersions that allow us to forget these compatibility conditions in the construction of the high dimensional holonomy map.
\end{remark}

\section{Ruled submanifolds in graded manifolds}
\label{sc:ruledsub}
In this section we consider a particular type of submanifolds for which the admissibility system reduces to a system of ODEs along the characteristic curves, that rule these submanifolds by determining their degree since the other adapted tangent vectors tangent to $M$ have highest degree equal to $s$.

\begin{definition}
\label{def:ruledimm}
Let $(N,\mh^1,\ldots,\mh^s)$ be an equiregular graded manifold of topological dimension $n$ and let  $\bar{M}$ a $m$-dimensional manifold with $m<n$. We say that an immersion  $\Phi: \bar{M} \to N$ is \emph{ruled} if 
\begin{equation}
\label{eq:degreerest}
\deg(M)=(m-1)s +\iota_0,
\end{equation}
 where $\iota_0 $ is the integer defined in \ref{eq:iota} satisfying $\dim(\tilde{\mh}^{\iota_0})=1$ and $M=\Phi(\bar{M})$. In this case, we will call the image of the immersion $M$ a ruled submanifold.\\
\end{definition}

Let $\bar{p}$ be a point in $\bar{M}$ such that $p=\Phi(\bar{p})$ is a point of maximum degree in $M$. Following the argument of Section~\ref{sc:admsys}, we consider an open neighborhood $\bar{O}$ of $\bar{p}$ such that $O=\Phi(\bar{O})$ is an embedding of fixed degree. Let $(E_1,\ldots,E_m)$ be an adapted basis to $TO$.
Therefore $\deg(E_1)=\iota_0$ and $\deg(E_j)=s$ for $j=2,\ldots,m$ and $k=n_{\iota_0}-1$.  Then we follow the construction described in Section \ref{sc:admsys} to provide the metric $g$ and the orthonormal basis $E_1,\ldots,E_m,V_{m+1},\ldots,V_n$  whose sorting is a local adapted basis of $TU$.
Since $\deg(E_j)\ge \deg(V_i)$ for each $j=2,\ldots,m$ and $i=m+k+1,\ldots,n$, the only derivative that appears in \eqref{eq:admsystem} is $E_1$. Therefore we deduce that a vector field $V^{\perp}$, given by equation \eqref{eq:perpvector}, is admissible if and only if it satisfies
\begin{equation}
E_1( f_{i})+\sum_{r=m+k+1}^n b_{i 1 r} f_r + \sum_{h=m+1}^{m+k} a_{ i 1 h} g_h=0,
\label{eq:ruledadm}
\end{equation}
for $i=m+k+1,\ldots,n$ and for each $q \in O$
\[
a_{ i 1 h}(q)=\escpr{ V_i(q), [E_1,V_h](q)},
\]
\[
b_{ i 1 r}(q)=\escpr{V_i(q), [E_1,V_r](q)}
\]
and $f_r \in C^1(O)$, $g_h \in C(O)$.
Given $p$ in $M$ each point $q$ in a local neighborhood $O$ of $p$ in $M$ can be reached using the exponential map as follows 
\begin{equation}
\label{eq:expexp}
q=\exp( x_1 E_1) \exp \left( \sum_{j=2}^m x_j E_j \right)(p).
\end{equation}
On this open neighborhood $O \subset M$ we consider the local coordinates $x=(x_1,x_2,\ldots,x_m)$ given by the inverse map $\Xi$ of the exponential map defined in \eqref{eq:expexp}. In the literature, these coordinates are commonly called  exponential or canonical coordinates of the second kind, see \cite{HH91,Bellaiche}. We set $\hat{x}:=(x_2,\dots,x_m)$. Given  a relative compact open subset $\Omega \subset \subset \Xi(O)$  we consider
\begin{equation}
\label{eq:Sigma0}
\Sigma_0=\{x_1=0\} \cap \Omega
\end{equation}
 be the $(m-1)$-dimensional leaf  normal to $E_1$.  Then there exists $\eps>0$ so that the closure of the cylinder 
\begin{equation}
\label{eq:cylinder}
\Omega_{\eps}=\{(x_1, \hat{x}) \ : \ 0< x_1<\eps, \ \hat{x} \in  \Sigma_0 \}
\end{equation}
is contained in $\Xi(O)$. Then $\Sigma_{\eps}=\{(\eps,\hat{x}) \ : \ \hat{x} \in \Sigma_0 \}$ is the top of the cylinder. Since $d \Xi(E_1)=\partial_{x_1}$ in this exponential coordinates of the second kind the \emph{admissibility system} \eqref{eq:ruledadm} is given by 
\begin{equation}
\label{eq:ruledadm2}
 \dfrac{\partial F(x)}{\partial x_1} =-B(x) F(x)-A(x) G(x),
\end{equation}
where we set 
\begin{equation}
\label{def:FG}
F=\begin{pmatrix} f_{m+k+1} \\ \vdots \\ f_n \end{pmatrix}, \quad G=\begin{pmatrix} g_{m+1} \\ \vdots \\ g_{m+k} \end{pmatrix}
\end{equation}
and we denote by $B$ the $(n-m-k)$ square matrix whose entries are $b_{i 1 r}$,  by $A$ the  $(n-m-k) \times k $ matrix whose entries are $a_{i 1 h}$.

\section{The high dimensional holonomy map for ruled submanifolds}
\label{sc:holonomy}
 For ruled submanifolds the system  \eqref{eq:admsystem} reduces to the system of ODEs \eqref{eq:ruledadm} along the characteristic curves.  Therefore, a uniqueness and existence result for the solution is given by the classical Cauchy-Peano Theorem, as in the case of curves in \cite[Section 5]{2019arXiv190204015C}.

Let $\Phi: \bar{M} \to N$ be a ruled immersion in a graded manifold. Let $\Omega_{\eps}$ be the open cylinder defined in \eqref{eq:cylinder} and $T_{\Sigma_0}(f)=f(0,\cdot)$ and $T_{\Sigma_{\eps}}(f)=f(\eps,\cdot)$ be the operators  that evaluate functions at $x_1=0$ and at $x_1=\eps$, respectively. 

Let $C_0(\Omega_{\eps})$  the Banach space of continuous functions on $\Omega_{\eps}$ vanishing at the infinity, that is the closure of the space of compactly supported function on  $\Omega_{\eps}$, see \cite[Theorem 3.17]{Rudin}.  We always consider for each  $ f \in C_0(\Omega_{\eps})$  the supremum norm
\[
\|  f \|_{\infty} =\sup_{x \in \Omega_{\eps}} |f(x)|.
\]
We will denote by $\bar{\Omega}_{\eps}$ the closure of the open set $\Omega_{\eps} \subset \rr^m$ and by $C(\bar{\Omega}_{\eps})$ the Banach space of continuous functions on the compact $\bar{\Omega}_{\eps}$. Then we consider the following Banach spaces:
\begin{enumerate}
\item $ \displaystyle \mathcal{H}_0( \Omega_{\eps})=\left\{ \sum_{h=m+1}^{m+k} g_h V_h \ : \  g_h \in  C_0( {\Omega}_{\eps})  \right\}$.
\item $\displaystyle \mathcal{V}^1(\bar{\Omega}_{\eps})=\left\{ \sum_{r=m+k+1}^{n} f_r V_r \ : \  \partial_{x_1} f_r \in C( \bar{\Omega}_{\eps}), f_r \in  C( \bar{\Omega}_{\eps}), T_{\Sigma_0}(f_r)=0 \right\}$.
\item $\displaystyle \mathcal{V}(\Sigma_{\eps})=\left\{ \sum_{r=m+k+1}^{n} f_r V_r \ : \  f_r \in  C_0 (\Sigma_{\eps}) \right\}$, where $C_0 (\Sigma_{\eps})$ is the space of continuous functions on $\Sigma_{\eps}$ vanishing at the infinity.
\end{enumerate}
Notice that the respective norms of these Banach spaces are given by
\begin{enumerate}
\item $\displaystyle \|  G \|_{\infty}= \max_{h=m+1,\ldots,m+k} \sup_{x \in \Omega_{\eps}} | g_h(x) | $ 
 \item $\displaystyle \|  F \|_{1}= \max_{r=m+k+1,\ldots,n}( \sup_{x \in \bar{\Omega}_{\eps}} | f_r (x)|+ \sup_{x \in \bar{\Omega}_{\eps}} | \partial_{x_1} f_r (x)| ) $ 
 \item $\displaystyle \|  F \|_{\infty, \Sigma_{\eps}}= \max_{r=m+k+1,\ldots,n} \sup_{\hat{x} \in \Sigma_{\eps}} | f_r(\hat{x} )| $,
\end{enumerate}
where $F$ and $G$ are defined in \eqref{def:FG}.

Therefore the existence and the uniqueness of the solution of the  Cauchy problem allows us to define the holonomy type map 
\begin{equation}
H_{M}^{\eps}: \mh_0(\Omega_{\eps}) \to \mathcal{V}(\Sigma_{\eps}),
\end{equation}
in the following way: we consider a horizontal compactly supported continuous vector field $$Y_{\mh}=\sum_{h=m+1}^{m+k} g_h V_h \in \mathcal{H}_0(\Omega_{\eps})  $$
and we fix  the initial condition $Y_V(0,\hat{x})=0$. Then there exists a unique solution  $$Y_{\mathcal{V}}=\sum_{r=m+k+1}^n f_r V_r \in \mathcal{V}^1(\Omega_{\eps})$$ of the admissibility system \eqref{eq:ruledadm2} with initial condition $Y_V (0,\hat{x})=0$. Letting 
\[
\mathbf{T}_{\Sigma_{\eps}}: \mathcal{V}^1(\Omega_{\eps}) \to \mathcal{V} (\Sigma_{\eps})
\]
be the evaluating operator for vertical vectors fields at $x_1=\eps$ defined by $\mathbf{T}_{\Sigma_{\eps}}(V)=V(\eps,\cdot)$, we define $H_{M}^{\eps}(Y_{\mh})=\mathbf{T}_{\Sigma_{\eps}}(Y_{\mathcal{V}})$.
\begin{definition}
We say that $\Phi$ restricted to $\bar{\Omega}_{\eps}$ is \emph{regular} if the image of the holonomy map $H_{M}^{\eps}$ is  a dense subspace of $\mathcal{V} (\Sigma_{\eps})$, that contains a normalized Schauder  basis of $\mathcal{V}(\Sigma_{\eps})$ (see \cite[Definition 14.2]{S70}) .
\end{definition}

The following result allows the integration of the differential system \eqref{eq:ruledadm2} to explicitly compute the holonomy map.

\begin{proposition}
\label{prop:inthol}
In the above conditions, there exists a square regular matrix $D(x_1,\hat{x})$ of order $(n-k-m)$ such that
\begin{equation}
\label{eq:F(b)}
F(\eps,\hat{x})=- D(\eps,\hat{x})^{-1}\int_0^{\eps} (D A)(\tau,\hat{x}) G(\tau,\hat{x}) \, d\tau,
\end{equation}
for each $\hat{x} \in \Sigma_0$.
\end{proposition}

\begin{proof}
Lemma~\ref{lm:det} below allows us to find a regular matrix $D(x_1,\hat{x})$ such that $\partial_{x_1} D= DB$. Then equation $\partial_{x_1}F=-BF-AG$ is equivalent to $\partial_{x_1}(DF)=- DAG$. Integrating between $0$ and $\eps$, taking into account that $F(0,\hat{x})=0$ for each $\hat{x} \in \Sigma_0 $, and multiplying by $D(\eps,\hat{x})^{-1}$, we obtain \eqref{eq:F(b)}.
\end{proof}

\begin{lemma}
\label{lm:det}
Let $E$ be an open set of $\rr^{m-1}$. Let $B(t,\lambda)$ be a continuous family of square matrices on $[0,\eps]\times E$. Let $D(t,\lambda)$ be the solution of the Cauchy problem
\[
\partial_{t}D(t,\lambda)=D(t,\lambda) B(t,\lambda)\ \text{on }[0,\eps]\times E,  \quad D(0,\lambda)=I_d,
\]
for each $\lambda \in E$.
Then $\det  D(t,\lambda)\ne0$ for each $(t,\lambda) \in [0,\eps]\times E$.
\end{lemma}
\begin{proof}
By the Jacobi formula we have 
\[
\dfrac{\partial (\det D(t,\lambda))}{\partial t}=\text{Tr}\left(\text{adj}\, D(t,\lambda)\, \dfrac{\partial D(t,\lambda)}{\partial t}\right),
\]
where $\text{adj} D$ is the classical adjoint (the transpose of the cofactor matrix) of $D$ and $\text{Tr}$ is the trace operator. Therefore
\begin{equation}
\label{eq:detode}
\dfrac{\partial (\det D(t,\lambda))}{\partial t}=\text{Tr}\left((\text{adj}\, D(t,\lambda)) D(t,\lambda) B(t,\lambda) \right)=\det D(t,\lambda)\, \text{Tr}(B(t,\lambda)).
\end{equation}
Since $\det D(0,\lambda)=1$, the solution for \eqref{eq:detode} is given by 
\[
 \det D(t,\lambda)=e^{\int_a^t \text{Tr}(B(\tau,\lambda)) \, d \tau}>0,
\]
for all $(t,\lambda) \in [0,\eps]\times E$. Thus, the matrix $D(t,\lambda)$ is invertible for each $(t,\lambda) \in [0,\eps]\times E$.
\end{proof}

\begin{definition}
We say that the matrix $\tilde{A}(x_1,\hat{x}):=(DA)(x_1,\hat{x})$ on $\Omega_{\eps}$ defined in Proposition \ref{prop:inthol} is linearly full $\rr^{n-m-k}$ if and only if for each $\hat{x} \in \Sigma_0 $
\[
\dim \left( \text{span} \left\{\tilde{A}^1(x_1,\hat{x}), \ldots, \tilde{A}^k(x_1,\hat{x})  \quad \forall \ x_1\in (0,\eps) \right\} \right)=n-m-k,
\]
where $\tilde{A}^i$ for $i=1,\ldots, k$ are the columns of $\tilde{A}(x_1,\hat{x})$.
\end{definition}

\begin{lemma}
\label{lm:operatorsurj}
Let $X$ be a Banach space and $Y \subset X$. Let $L: X \to \rr$ be a bounded linear functional, $L \not\equiv 0$ such that $L(x)=0$ for each $x \in Y$. Then $Y$ is not dense in $X$.
\end{lemma}
\begin{proof}
Fix $y$ in $\bar{Y}$. Then there exists $\{y_n\}_{n \in \nn} \subseteq Y$ such that  $y_n \to y$ as $n \to +\infty$. Since $L$ is continuous we have $L(y_n)\xrightarrow{n \to +\infty} L(y)$. On the other hand, by assumption $L(y_n)=0$, then we conclude that  $L(y)=0$. Therefore we have $L(y)=0$ for each $y \in \bar{Y}$. Assume by contradiction  that $Y$ is dense in $X$, i.e. $\bar{Y}=X$. Therefore we have $L(x)=0$ for each $x \in X$, that implies $L\equiv0$, that is absurd.
\end{proof}

\begin{proposition}
\label{prop:linearlyfull}
The immersion $\Phi$ restricted to $\bar{\Omega}_{\eps}$ is regular if and only if $\tilde{A}(x_1,\hat{x})$ is linearly full in $\rr^{n-m-k}$.
\end{proposition}
\begin{proof}
Assume that immersion $\Phi$ restricted to $\bar{\Omega}_{\eps}$ is not regular. Then the closure of image of the holonomy map  $\overline{\text{Range}(H_{M}^{\eps})}$ is a proper closed subspace of $\mathcal{V} (\Sigma_{\eps})$. By \cite[Corollary 1.8]{Brezis}  there exists  $\mu \in (\mathcal{V} (\Sigma_{\eps}))^*$, $\mu \not\equiv 0$ such that $\mu(F)=0$ for each $F \in \text{Range}(H_{M}^{\eps})$, where each element $F$ in $\text{Range}(H_{M}^{\eps})$ is given by the representation formula \eqref{eq:F(b)}.
Thanks to Riesz's Theorem (see for instance \cite[Theorem 4.7]{Maggi12}, \cite[Chapter 7]{Folland84} or \cite{BL74}) the total variation $|\mu|$ is a Radon measure on $\Sigma_{\eps}$ and there exists $|\mu|$-measurable function $\Gamma: \Sigma_{\eps} \to \rr^{n-m-k}$ with $|\Gamma|=1$ $|\mu|$-a.e. such that
\begin{equation}
\label{eq:vanishingintforallG}
\begin{aligned}
0= \mu ( F(\eps,\cdot ))&=-\mu \left(  D(\eps,\hat{x})^{-1}\int_0^{\eps}  (D A)(\tau,\hat{x}) G(\tau,\hat{x}) \, d\tau \right)\\
&=- \int_{\Sigma_{\eps}} \Ga(\hat{x}) D(\eps,\hat{x})^{-1}   \left(\int_0^{\eps}  (D A)(\tau,\hat{x}) G(\tau,\hat{x}) d\tau \right) d|\mu|(\hat{x}) \\
&=- \int_{\Sigma_{\eps}} \tilde{\Ga}(\hat{x})  \left(\int_0^{\eps}  (D A)(\tau,\hat{x}) G(\tau,\hat{x}) d\tau \right) d|\mu|(\hat{x}) \\
\end{aligned}
\end{equation}
where $\tilde{\Ga}=\Ga(\hat{x}) D(\eps,\hat{x})^{-1} \neq 0$ a.e. w.r.t. to $|\mu|$. As this formula \eqref{eq:vanishingintforallG} holds for any $G(\tau,\hat{x}) \in C_0(\Omega_{\eps})$ we can consider $G(\tau, \hat{x})= h(\hat{x}) g(\tau) $ for each $g \in C_0((0,\eps),\rr^k)$ and $h \in C_0( \Sigma_{\eps},\rr)$. We notice that $\tilde{\Gamma} (\hat{x}) \int_0^{\eps} (D A)(\tau,\hat{x}) g(\tau) \, d\tau \in L^1_{\text{loc}} (\Sigma_{\eps},|\mu|)$. Then, by Lemma \ref{lm:intRadon} and the fundamental lemma of the Calculus of Variations for continuous functions we deduce that $\tilde{\Ga}(\hat{x})\tilde{A}(\tau,\hat{x})=0$ for all $\tau \in (0,\eps)$ and $|\mu|$-a.e. in $\hat{x} \in \Sigma_{\eps}$. 

Since the $\text{supp}(|\mu|) \ne \emptyset $ for each $\hat{x} \in \text{supp}(|\mu|)$ there exists an open neighborhood $N_{\hat{x}}$ of $\hat{x}$ such that $|\mu|(N_{\hat{x}})>0$. Eventually reducing $N_{\hat{x}}$ we can assume $|\mu|(N_{\hat{x}})<+\infty$, by the locally finite property of the Radon measure. Therefore by the Lusin's Theorem \cite[Chapter 7]{Folland84} for every $\epsilon>0$ there exists a compact $K \subset N_{\hat{x}} $ such that $|\mu|(N_{\hat{x}} \smallsetminus K) < \epsilon$ and $\tilde{\Ga}|_{K}$ is continuous. Considering $\epsilon<|\mu|(N_{\hat{x}})$ we obtain $|\mu|(K)>0$, therefore there exists $\hat{x}_0 \in K$ such that  $\tilde{\Ga}(\hat{x}_0 ) \tilde{A}(t,\hat{x}_0 )=0$ for each $t \in (0,\eps) $. Then the columns of $\tilde{A}(t,\hat{x}_0) $  are contained in the hyperplane of $\rr^{n-m-k}$ determined by $\tilde{\Ga}(\hat{x}_0 )$.  Identifying the open set $\Sigma_{\eps}$ with the open set $\Sigma_0$, by the map $(\eps,\hat{x}) \to (0,\hat{x})$,  we deduce that $\tilde{A}$ is not linearly full.

Conversely, assume that $\tilde{A}$ is not linearly full. Then there exist a point $\hat{x}_0 \in \Sigma_0$ and a row vector with $(n-m-k)$ coordinates $\Ga\neq 0$ such that $\Ga\tilde{A}(x_1,\hat{x}_0)=0$ for all $x_1\in (0,\eps)$. Then, denoting by $\delta_{\hat{x}_0}(\varphi)= \varphi(\hat{x}_0)$ the delta distribution and $\tilde{\Ga}=\Ga D(\eps, \hat{x}_0)\neq0 $, we have 
\[
\tilde{\Ga} \delta_{\hat{x}_0} ( F(\eps,\cdot))=- \int_0^{\eps} \Ga (D A)(\tau, \hat{x}_0) G(\tau,\hat{x}_0) \, d\tau=0.
\]
Since the vector-value Radon measure $\tilde{\Ga} \delta_{\hat{x}_0} \not\equiv 0$ annihilates the image of the holonomy map, by Lemma \ref{lm:operatorsurj} we conclude that the image of holonomy map is not a dense subspace of $\mathcal{V}(\Sigma_{\eps})$.
\end{proof}

For the reader’s convenience, in Lemma \ref{lm:intRadon} we recall  a classical result of calculus of variations, see for instance \cite[Corollary 4.24]{Brezis} or \cite[Exercise 4.14]{Maggi12}.
\begin{lemma}
\label{lm:intRadon}
Let $\Omega$ be an open subset of $\rr^n$ and $\mu$ be a Radon measure on $\Omega$. If $f:\Omega \to \rr$ is a measurable function in $L^1_{\text{loc}}(\Omega,\mu)$ such that 
\[
\int_{\Omega} f(x) h(x) d\mu(x)=0 \quad \forall h \in C_0(\Omega),
\]
then $f=0$ a.e. w.r.t. $\mu$.
\end{lemma}
\begin{proof}
First of all we claim that for each compact set $K \subset \Omega$
\[
\int_K f(x) d\mu(x)=0.
\]
Fix a compact $K \subset \Omega$ and  consider  a sequence of continuous compactly supported functions $h_n \le 1$ on $\Omega$, $h_n \equiv 1$ on $K$, vanishing out of small open neighborhood $U$ of $K$  such that $\text{supp}(h_{n+1}) \subset \text{supp}(h_{n}) $ for each $n \in \nn$ and $h_n (x) \xrightarrow[]{n\to +\infty} \chi_K(x)$ for all $x \in \Omega$, where 
\[
\chi_K(x)=\begin{cases}
1 & \quad \text{if} \quad  x \in K\\
0 &  \quad \text{if} \quad x \in \Omega \setminus K.
\end{cases}
\] 
Since we have the pointwise convergence and $|f(x) h_n(x)|\le |f(x)|$ for each $n \in \nn$ with $f \in L^1(\text{supp}(h_1),\mu)$, by the dominated convergence theorem  we obtain 
\[
0=\int_{\Omega} f(x) h_n(x) d\mu(x) \xrightarrow[]{n\to +\infty} \int_{\Omega} f(x) \chi_K(x) d\mu(x)= \int_K f(x) d\mu(x).
\]
Let us consider $\delta >0$ and the Borel sets
\[
E_{\delta}^{+}=\{x \in \Omega \ : \ f(x)> \delta\}=f^{-1}((\delta,+\infty))
\]
and
\[
E_{\delta}^{-}=\{x \in \Omega \ : \ f(x)<- \delta\}=f^{-1}((-\infty,-\delta)).
\]
Then $\mu(E_{\delta}^{+})=\sup_{K \subset E_{\delta}^{+} } \mu(K)$ for each compact set $K  \subset E_{\delta}^{+}$.
On the other hand we  have
\[
0=\int_K f(x) d\mu(x)\ge \delta \mu(K).
\]
Therefore $\mu(K)=0$ for each $K \subset E_{\delta}^{+} $, then $\mu(E_{\delta}^{+})=0$. Hence  as $\delta \to 0$ we obtain $\mu(E^{+})=0$,
where 
\[
E^{+}=\{x \in \Omega \ : \ f(x)> 0\}.
\]
A similar argument prove that  $\mu(E^{-})=0$, where $E^{-}=\{x \in \Omega \ : \ f(x)< 0\}$.
\end{proof}

The following result provides a useful characterization of non-regularity
\begin{theorem}
\label{th:singchar}
The immersion $\Phi$ restricted to $\bar{\Omega}_{\eps}$ is non-regular if and only if there exist a point $\hat{x}_0 \in \Sigma_0$ and a row vector field $\Lambda(x_1,\hat{x}_0)\ne0$  for all $x_1 \in [0,\eps]$  that solves the following system
\begin{equation}
 \begin{cases}
 \partial_{x_1} \Lambda(x_1,\hat{x}_0)= \Lambda(x_1,\hat{x}_0) B(x_1,\hat{x}_0)\\
 \Lambda(x_1,\hat{x}_0) A(x_1,\hat{x}_0)=0.
 \end{cases}
 \label{eq:singularsys}
\end{equation}
\end{theorem}
\begin{proof}
 
Assume that $\Phi$ restricted to $\bar{\Omega}_{\eps}$ is non-regular, then by Proposition~\ref{prop:linearlyfull} there exist a point $\hat{x}_0 \in \Sigma_0$ and  a row vector $\Gamma \ne 0$ such that 
\[
 \Gamma D(x_1,\hat{x}_0) A(x_1,\hat{x}_0)=0 
\] 
for all $x_1\in [0,\eps]$, where $D(x_1,\hat{x}_0)$ solves 
\begin{equation}
\label{eq:homD}
\begin{cases}
\partial_{x_1} D= D B\\
D(0,\hat{x}_0)=I_{n-m-k}.
\end{cases}
\end{equation}
Since $\Gamma$ is a constant vector and $D(x_1,\hat{x}_0)$ is a regular matrix by Lemma~\ref{lm:det} , $\Lambda(x_1,\hat{x}_0):=\Gamma D(x_1,\hat{x}_0)$ solves the system \eqref{eq:singularsys} and  $\Lambda(x_1,\hat{x}_0) \ne 0$ for all $x_1 \in [0,\eps]$.

Conversely, any solution of the system \eqref{eq:singularsys} is given by 
\[
\Lambda(x_1,\hat{x}_0)= \Gamma D(x_1,\hat{x}_0),
\]
where $\Gamma=\Lambda(0,\hat{x}_0)\ne 0$ and $D(x_1,\hat{x}_0)$ solves the equation \eqref{eq:homD}.
Indeed, let us consider a general solution $\Lambda(t,\hat{x}_0)$ of \eqref{eq:singularsys}. If we set
\[
\Psi_{\hat{x}_0}(t)=\Lambda(t,\hat{x}_0)-\Gamma D(t,\hat{x}_0),
\]
where $\Gamma=\Lambda(0,\hat{x}_0)\ne 0$ and $D(t,\hat{x}_0)$ solves the equation \eqref{eq:homD}, then we deduce 
\[
\begin{cases}
\partial_t \Psi_{\hat{x}_0}(t)=\Psi _{\hat{x}'}(t)B(t,\hat{x}_0)\\
\Psi_{\hat{x}_0}(0)=0.
\end{cases}
\]
Clearly the unique solution of this system is $\Psi_{\hat{x}_0}(t)\equiv 0$.
Hence we  conclude that $\Gamma \tilde{A}(x_1,\hat{x}_0)=0$. Thus $\tilde{A}(x_1,\hat{x}_0)$  is not fully linear and by Proposition~\ref{prop:linearlyfull}  we are done.
\end{proof}

\section{Integrability of admissible vector fields for a ruled regular submanifold}
\label{sc:intheo}
In this section we deduce the main result Theorem~\ref{thm:intcriterion}. As we pointed out in the Introduction we need that the space of  simple $m$-vectors of degree grater than $\deg(M)$ is quite simple. Therefore we give the following definition.

\begin{definition}
\label{def:FTG}
We say that a $m$-dimensional ruled immersion, see Definition \ref{def:ruledimm}, $\Phi: \bar{M} \to N$ into an equiregular graded manifold $(N,\mh^1,\ldots,\mh^s)$
\begin{enumerate}
\item[(i)]  \emph{fills the grading from the top} if $n_{s}-n_{s-1}=m-1$,  where $n_s=\dim(\mh^s)$ and $n_{s-1}=\dim(\mh^{s-1})$;
\item[(ii)] is \emph{foliated by curves} of degree grater than or equal to $s-3$ if $ \iota_0\ge s-3$.
\end{enumerate} 
A ruled submanifold verifying $(i)$ and $(ii)$ will be called a \emph{FGT-$(s-3)$} ruled submanifold and in this case $(ii)$ is equivalent to
\begin{equation}
\label{eq:restons}
s-3\le \iota_0\le s-1.
\end{equation}

\end{definition}

\begin{remark}
\label{rk:mdvectors}
Since $n_{s}-n_{s-1}=m-1$ and the condition \eqref{eq:restons} holds we have that the only simple $m$-vectors of degree strictly grater than $\deg(M)$ are
\[
V_i \wedge E_2 \wdw E_m
\]
for $i=m+k+1,\ldots,n$. When $\iota_0=s-1$ the submanifold has maximum degree therefore all vector fields are admissible, thus there are no singular submanifold.
\end{remark} 
Keeping the previous notation we now consider the following spaces
\begin{enumerate}
\item $ \quad \displaystyle \mathcal{H}(\Sigma_0)=\left\{ Y_{\mh}=\sum_{i=m+1}^{m+k} g_i V_i \ : \  g_i \in  C(\bar{\Omega}_{\eps}), T_{\Sigma_0}(g_i)=0 \right\} $ where the norm is given by 
\[
\| Y_{\mh} \|_{\infty}:= \max_{i=m+1,\ldots,m+k}\sup_{x \in \bar{\Omega}_{\eps} } |g_i |
\]
\item $ \quad  \displaystyle \mathcal{V}^1(\Sigma_0)=\left\{ Y_{\mathcal{V}}=\sum_{i=m+k+1}^{n} f_i V_i \ : \  \partial_{x_1} f_i \in C( \bar{\Omega}_{\eps}), f_i \in  C( \bar{\Omega}_{\eps}), T_{\Sigma_0}(f_i)=0 \right\},$  where the norm is given by 
\[
\| Y_{\mathcal{V}} \|_1:= \max_{i=m+k,\ldots,n}( \sup_{x \in \bar{\Omega}_{\eps} }  |f_i |+\sup_{x \in \bar{\Omega}_{\eps} } |\partial_{x_1} f_i | )
\]
\item $\Lambda ( \Sigma_0)$ is the set of elements given by
\[
\sum_{i=m+k+1}^{n}  z_{i}(x_1,\ldots,x_m) \  V_i \wedge E_2 \wdw E_m \ 
\]
where  $ z_{i} \in C( \bar{\Omega}_{\eps})$ vanishing on $\Sigma_0$.
\end{enumerate}
We denote by ${\Pi}_d$ the orthogonal projection over the space $\Lambda(\Sigma_0)$, that is the bundle over the vector space of simple $m$-vectors of degree strictly grater than $d$, thanks to Remark \ref{rk:mdvectors}. Then we set
\begin{equation}
\label{eq:defG}
{\mathcal{G}}:\mathcal{H}(\Sigma_0)\times \mathcal{V}^1(\Sigma_0) \to \mathcal{H}(\Sigma_0)\times {\Lambda}(\Sigma_0),
\end{equation}
defined by
\[
{\mathcal{G}}(Y_1,Y_2)=(Y_1,{\mathcal{F}}(Y_1+Y_2)),
\]
where
 \begin{equation}
 \label{eq:mathcalF}
{\mathcal{F}}(Y)={\Pi}_d \left( d\Gamma(Y)(E_1) \wedge \ldots \wedge d\Gamma(Y)(E_m)\right)
\end{equation}
and $\Gamma(Y)(q)=\exp_{q}(Y_q)$ for each $q \in O$. The open set $O$ is defined in Section~\ref{sc:admsys} and here $\exp$ denotes the Riemannian exponential map  defined by means of the geodesic flow on $TN$ induced by the Riemannian metric $\escpr{\cdot,\cdot}$ (see \cite[Chapter 3]{Carmo}). In equation \eqref{eq:mathcalF} we consider $E_j$ for each $j=1,\ldots,m$ as vector fields restricted to $O$ (to be exact we should use $\tilde{E}_j$ following the notation introduced in Section \ref{sc:admsys}) and $d\Gamma(Y)$ denotes the differential of $\Gamma(Y)$. Thanks to the diffeomorphism $\Xi$ defined in Section~\ref{sc:ruledsub} we can read the map $\mathcal{F}(Y)$ and the variation $\Gamma(Y)$ in exponential coordinates of the second kind $(x_1,x_2,\ldots,x_m)$ where the open cylinder $\Omega_{\eps}$ lives.

Observe that now ${\mathcal{F}}(Y)=0$ implies that the degree of the variation $\Gamma(Y)$ is less than or equal to $d$. Then
\[
D \mathcal{G} (0,0)(Y_1,Y_2)=(Y_1,D \mathcal{F}(0)(Y_1+Y_2)),
\]
where $D  \mathcal{F}(0)Y$ is given by
\[
\dfrac{d}{dt}\Big|_{t=0} \mathcal{F}(0+tY)=\sum_{i=m+k+1}^{n} \dfrac{d}{dt}\Big|_{t=0} \lambda_i (t)  \,V_i\wedge E_2 \wdw  E_m,
\]
with 
\begin{align*}
\dfrac{d}{dt}\Big|_{t=0}\lambda_i (t) =& \dfrac{d}{dt}\Big|_{t=0} \escpr{ d\Gamma(t Y)(E_1) \wedge \ldots \wedge d\Gamma(t Y)(E_m) ,(V_i\wedge E_2 \wdw  E_m)_{\Gamma(tY)}}\\
=&\sum_{j=1}^n \escpr{E_1\wdw \nabla_{E_j} Y \wdw E_m, V_i\wedge E_2 \wdw  E_m}+\\
&+\escpr{E_1\wdw E_m, \nabla_{Y} (V_i\wedge E_2 \wdw  E_m) },
\end{align*}
that is the right hand side of the equation \eqref{eq:system of PDEs for admissible}.
Therefore, following the computations developed  in Section~\ref{sc:ruledsub}  and using the exponential coordinates of the second kind we have  
\[
D \mathcal{F}(0) Y=\sum_{i=m+k+1}^{n} \Big(  \dfrac{\partial f_i(x)}{\partial x_1} +\sum_{r=m+k}^n b_{i1r} f_r+ \sum_{h=m+1}^{m+k}a_{i1h} g_h\Big) V_i\wedge E_2 \wdw  E_m.
\]
on $\Omega_{\eps} \subset \Xi(O) $, defined in \eqref{eq:cylinder}.
Observe that $D \mathcal{F}(0)Y=0$ if and only if $Y$ is an admissible vector field, namely $Y$ solves \eqref{eq:ruledadm2}. Moreover, we have that $A$ and $B$ are bounded the supremum norm on $\Omega_{\eps}$, since they are continuous  on $\Xi(O)$ and bounded on the compact $\bar{\Omega}_{\eps}$.\\
Our objective now is to prove that the map $D \mathcal{G}(0,0)$ is an isomorphism of Banach spaces. To show this, we shall need the following result.

\begin{proposition}
\label{prop:Gisomor}
The differential $D \mathcal{G}(0,0)$ is an isomorphism of Banach spaces.
\end{proposition}
\begin{proof}
We first observe that $D \mathcal{G}(0,0)$ is injective, since $D \mathcal{G}(0,0)(Y_1,Y_2)=(0,0)$ implies that $Y_1=0$ and that the vertical vector field $Y_2$ satisfies the compatibility equations with initial condition $Y_2(0,\hat{x})=0$ for each $\hat{x} \in \Sigma_0$. Hence $Y_2=0$. The map $D \mathcal{G}(0,0)$ is continuous.
Indeed, if for instance we consider the $1$-norm on the product space we have
\begin{align*}
 \| D \mathcal{G}(0,0)(Y_1, Y_2) \|&= \| (Y_1, D\mathcal{F}(0)(Y_1+Y_2)) \|\\
                        &\le \| Y_1 \|_{\infty} + \|D \mathcal{F}(0)(Y_1+Y_2)) \|_{\infty}\\
                        &\le (1+ \|(a_{hij}) \|_{\infty})\|Y_1\|_{\infty}+ (1+\|(b_{rij}) \|_{\infty})\|Y_2 \|_{1}.\\                        
\end{align*} 

To show that $D \mathcal{G}(0,0)$ is surjective, we take $(Y_1,Y_2)$ in the image, and we find a vector field $Y$ on $\Omega_{\eps}$ such that $Y_{\mh}=Y_1$, $D\mathcal{F}(0)(Y)=Y_2$ and $Y_{\mathcal{V}}(0,\hat{x})=0$. The map $D \mathcal{G}(0,0)$ is open because of the estimate \eqref{eq:DGopen} given in Lemma~\ref{lem:odeestimate} below.
\end{proof}

\begin{lemma}
\label{lem:odeestimate}
In the above conditions, assume that $D\mathcal{F}(0)(Y)=Y_2$ and $Y_{\mh}=Y_1$ and $Y(a)=0$. Then there exists a constant $K$ such that 
\begin{equation}
\label{eq:DGopen}
 \| Y_{\mathcal{V}} \|_{1} \le K (\|  Y_2 \|_{\infty}+ \| Y_1  \|_{\infty})
\end{equation}
\end{lemma}

\begin{proof}
We write 
\[
Y_1=\sum_{h=m+1}^{m+k} g_h V_h, \quad Y_2=\sum_{i=k+1}^n z_i \, V_i \wedge E_2 \wdw E_m \quad \text{and} \quad Y_{\mathcal{V}}=\sum_{r=k+1}^n f_r V_r.
\]
Then $Y_v$  is a solution of the ODE given  by
\begin{equation}
\label{eq:ODEprop}
 \partial_{x_1} F(x_1,\hat{x})=-B(x)F(x_1,\hat{x})+ Z(x_1,\hat{x})- A(x) G(x_1,\hat{x})
\end{equation}
where $B(x), A(x)$ are defined after \eqref{def:FG}, $F$, $G$ are defined in \eqref{def:FG} and we set 
\[
Z=\begin{pmatrix} z_{m+k+1} \\ \vdots \\ z_n \end{pmatrix}.
\]
Since $Y_{\mathcal{V}}(0,\hat{x})=0$ an $Y_{\mathcal{V}}$ solves \eqref{eq:ODEprop} in $(0,\eps)$, by Lemma \ref{lm:ODE ineq} there exists a constant $K$ such that 
\begin{equation}
\begin{aligned}
 \| Y_{\mathcal{V}} \|_{1}=\| F \|_{1} &\le K \|  Z(x)- A(x) \ G(x)  \|_{\infty}\\
               &\le \tilde{K} (\|  Y_2 \|_{\infty}+ \| Y_1  \|_{\infty}).
\end{aligned}
\end{equation}
where $\tilde{K}=K\max\{1, \| A(x)\|_{\infty}\}$.
\end{proof}
\begin{lemma}
\label{lm:ODE ineq}
Let $E$ be an open set of $\rr^{m-1}$. Let $u:[0,\eps] \times E \to \rr^d$ be the solution  of the inhomogeneous problem 
\begin{equation}
  \begin{cases}
  u'(t, \lambda)= A(t, \lambda) u(t,\lambda)+ c(t, \lambda),\\
  u(0,\lambda)=u_0( \lambda)
  \end{cases}
  \label{eq:inhsystem}
\end{equation}
where $A(t,\lambda)$ is a  $d\times d $ continuos matrix, bounded in the supremum norm on $[0,\eps] \times E$ and $c(t,\lambda)$ a continuos vector field bounded in the supremum norm  on $[0,\eps] \times E$. We denote by $u'$ the partial derivative $\partial_t u$ . Then,  there exists a constant $K$ such that 
 \begin{equation}
  \| u \|_{1}:= \| u \|_{\infty}+\|u'\|_{\infty} \leqslant K (\| c \|_{\infty}+ |u_0|_{\infty}).
  \label{dis:lemma ineq}
 \end{equation} 
\end{lemma}
\begin{proof}
We start from the case $r=1$. By \cite[Lemma 4.1] {Hartman} it follows 
\[
 u(t,\lambda)\leqslant \left(|u_0(\lambda)|+ \int_{0}^{t} |c(s,\lambda)| ds \right) e^{|\int_{0}^{t} \|A(s,\lambda) \| ds |},
\]
for each $\lambda \in E$ and where the norm of $A$ is given by $\sup_{|x|=1}|A \ x|$.
Therefore we have 
\begin{equation}
\label{dis:u}
 \sup_{t \in [0,\eps]} \sup_{\lambda \in E} |u(t,\lambda)| \le C_1  (\sup_{t \in [0,\eps]}  \sup_{\lambda \in E}  |c(t,\lambda)|+  \sup_{\lambda \in E} |u_0(\lambda)|) ,
\end{equation}
where we set
\[
 C_1=\eps e^{\eps \sup_{t \in [0,\eps] }\sup_{\lambda \in E} \|A(t, \lambda) \| }.
\]
Since $u$ is a solution of \eqref{eq:inhsystem} it follows 
\begin{equation}
\begin{aligned}
 \sup_{t \in [0,\eps]} \sup_{\lambda \in E} |u'(t,\lambda)|&\le \sup_{t \in [0,\eps]} \sup_{\lambda \in E} \|A(t,\lambda) \| \sup_{t \in [0,\eps]} \sup_{\lambda \in E} |u(t,\lambda)|+  \sup_{t \in [0,\eps]} \sup_{\lambda \in E} |c(t,\lambda)|\\
 &\le ( C_2 + 1) \sup_{t \in [0,\eps]} \sup_{\lambda \in E} |c(t,\lambda)|.\\
\end{aligned}
\label{dis:u'}
\end{equation}
Hence by \eqref{dis:u} and \eqref{dis:u'} we obtain 
\[
 \| u \|_1 \le K (\| c \|_{\infty} +\|u_0\|_{\infty}).
 \qedhere
\]
\end{proof}
Finally, we use the previous constructions to give a criterion for the integrability of admissible vector fields along a horizontal curve.

\begin{theorem}
\label{thm:intcriterion}
Let $\Phi: \bar{M} \to N$ be a ruled FGT-$(s-3)$  immersion into an equiregular graded manifold $(N,\mathcal{H}^1,\ldots,\mh^s)$ such that $\deg(M)=(m-1)s +\iota_0$, where $m=\dim(\bar{M})$, and $(i)$ and $(ii)$ in \ref{def:FTG}  hold. Let $\Omega_{\eps}=\{(x_1,\hat{x}) \ : \ 0< x_1<\eps, \ \hat{x} \in  \Sigma_0 \}$ with $\Sigma_0$ defined in \eqref{eq:Sigma0}. Assume that $\Phi$ is regular on the compact $\bar{\Omega}_{\eps}$. Then every admissible vector field with compact support in $\Omega_{\eps}$ is integrable.
\end{theorem}

\begin{proof}
If $\iota_0=s-1$ all vector fields are admissible, then all immersions are automatically regular. Each vector field $V$ is integrable for instance by the exponential map $\Gamma_t=\exp(tV)$.

Let now $s-3\le \iota_0 \le s-2$. Let us take $V$ vector field on $\Omega_{\eps}$ and $\{V^i\}_{i=1}^{\infty}$ vector fields equi-bounded in the supremum norm on $\bar{\Omega}_{\eps}$. Let $l^1(\rr)$  the Banach space of summable sequences. We consider the map
\[
\tilde{\mathcal{G}}:\big[(-\eps,\eps)\times l^1 (\rr)\big]\times \mathcal{H}(\Sigma_0)\times \mathcal{V}^1(\Sigma_0) \to \mathcal{H}(\Sigma_0)\times {\Lambda}(\Sigma_0),
\]
given by
\[
\tilde{\mathcal{G}}((\tau,(\tau_i),Y_1,Y_2))=(Y_1,\mathcal{F}(\tau V+\sum_{i=1}^{\infty}\tau_iV^i+Y_1+Y_2)),
\]
where $\mathcal{F}$ is defined in \eqref{eq:mathcalF}.
The map $\tilde{\mathcal{G}}$ is continuous with respect to the product norms (on each factor we put the natural norm, the Euclidean one on the interval, the $l^1$ norm and $||\cdot||_{\infty}$ and $||\cdot||_{1}$ in the spaces of vectors on $\Omega$). Moreover
\[
\tilde{\mathcal{G}}(0,0,0,0)=(0,0),
\]
since the immersion $\Phi$ has degree equal to $d$. Denoting by $D_Y$ the differential with respect to the last two variables of $\tilde{\mathcal{G}}$ we have that
\[
D_Y\tilde{\mathcal{G}}(0,0,0,0)(Y_1,Y_2)=D\mathcal{G}(0,0)(Y_1,Y_2)
\]
is a linear isomorphism thanks to Proposition~\ref{prop:Gisomor}. We can apply the Implicit Function Theorem to obtain maps
\[
Y_1:(-\eps,\eps)\times l^1 (\eps)\to \mathcal{H}(\Sigma_0), \quad Y_2:(-\eps,\eps)\times l^1 (\eps)\to \mathcal{V}^1(\Sigma_0),
\]
such that $\tilde{\mathcal{G}}(\tau,(\tau_i),(Y_1)(\tau,\tau_i),(Y_2)(\tau,\tau_i))=(0,0)$. We denote by $l^1 (\eps)$ the ball of radio $\eps$ in Banach space $l^1(\rr)$.  This implies that $(Y_1)(\tau,(\tau_i))=0$ and that
\begin{equation}
\label{eq:mcFIFT}
\mathcal{F}(\tau V+\sum_{i=1}^{\infty} \tau_iV^i+Y_2(\tau,\tau_i))=0.
\end{equation}
Hence the submanifolds
\[
\Gamma(\tau V+\sum_i \tau_i V^i+Y_2(\tau,\tau_i))
\]
have degree equal to or less than $d$.

Now we assume that $V$ is an admissible vector field compactly supported on $\Omega_{\eps}$, and that $V^i$ are admissible vector fields such that $V_{\mathcal{V}}^i$ vanishing on $\Sigma_0$. Then, differentiating \eqref{eq:mcFIFT}, we obtain that the vertical vector fields
\[
\frac{\ptl Y_2}{\ptl \tau}(0,0),\frac{\ptl Y_2}{\ptl \tau_i}(0,0)
\]
on $\Omega_{\eps}$ are admissible. Since they are admissible and vertical vector fields vanishing at $(0,\hat{x})$, they are identically $0$.

Since the image of the holonomy map is dense and contains a  normalized Schauder  basis for $\mathcal{V}(\Sigma_{\eps})$, we choose $\{V^i\}_{i=1}^{\infty}$ on $\Omega_{\eps}$ such that $\{\mathbf{T}_{\Sigma_{\eps}}(V^i_{\mathcal{V}})\}_{i \in \nn}$ is a  normalized Schauder  basis for $\mathcal{V}(\Sigma_{\eps})$. Then we consider the map
\[
\mathcal{P}:(-\eps,\eps)\times l^1 (\eps) \to \mathcal{C}_0(\Sigma_{\eps}, N)
\]
given by
\[
(\tau,(\tau_i))\mapsto \Gamma(\tau V +\sum_{i=1}^{\infty} \tau_i V^i+Y_2(\tau,\tau_i))_{|\Sigma_{\eps}},
\]
where $\mathcal{C}_0(\Sigma_{\eps}, N)$ is the set of continuous functions from $\Sigma_{\eps}$ to $N$ vanishing at infinity, that inherits its differential structure as submanifold of the Banach space  $C_0(\Sigma_{\eps}, \rr^{2n})$, see \cite[Section 5]{PT01}. For $s, (s_i)$ small, the image of this map is an infinite-dimensional submanifold 
of $\mathcal{C}_0(\Sigma_{\eps}, N)$ with tangent space at $\Phi|_{\Sigma_{\eps}}$ given by the Banach space $\mathcal{V}(\Sigma_{\eps})$ (as  $\mathbf{T}_{\Sigma_{\eps}}(V)=0$ and $\mathbf{T}_{\Sigma_{\eps}}(V^i)=\mathbf{T}_{\Sigma_{\eps}} (V_{\mathcal{V}}^i)$
generate $\mathcal{V}(\Sigma_{\eps})$ ).
Notice that 
\[
 \dfrac{\partial \mathcal{P}(0,0)}{\partial \tau_i}=\mathbf{T}_{\Sigma_{\eps}}(V^i)=\mathbf{T}_{\Sigma_{\eps}} (V_{\mathcal{V}}^i),
\]
for each $i \in \nn$.
Therefore the differential  $D_2 \mathcal{P}(0,0): l^1(\rr) \to \mathcal{V}(\Sigma_{\eps})$ defined by 
\[
D_2 \mathcal{P}(0,0)(\alpha)=\sum_{i=1}^{\infty} \alpha_i \mathbf{T}_{\Sigma_{\eps}} (V_{\mathcal{V}}^i)
\]
is injective, surjective and continuous. Then, by \cite[Corollary 2.7]{Brezis} $D_2 \mathcal{P}(0,0)$ is a Banach space isomorphism.
Moreover, we have
\[
 \dfrac{\partial \mathcal{P} (0,0)}{\partial \tau}=\mathbf{T}_{\Sigma_{\eps}}(V)=0,
\]
since $V$ is compactly supported in $\Omega_{\eps}$. Hence we can apply the Implicit Function Theorem to conclude that there exist $\eps'<\eps$ and a family of smooth functions $\tau_i(\tau)$, with $\sum_i |\tau_i(\tau)| < \eps$ for all $\tau \in (-\eps',\eps')$, so that
\[
\Gamma(\tau V+\sum_i \tau_i(\tau) V^i+Y_2(\tau,\tau_i(\tau)))
\]
takes the value $\Phi(\bar{p})$ for  each $\bar{p} \in \Sigma_{\eps}$. Since the vector fields $\{V^i\}_{i=1}^{\infty}$ are equi-bounded in the supremum norm on $\bar{\Omega}_{\eps}$, the series $\sum_i \tau_i(\tau) V^i$ is absolutely convergent on $\bar{\Omega}_{\eps}$. 

Clearly, we have
\[
 \mathcal{P}(\tau,(\tau_i(\tau)))(\bar{p})=\Phi(\bar{p}),
\]
for each  $\bar{p} \in \Sigma_{\eps}$.
Differentiating with respect to $\tau$ at $\tau=0$ we obtain 
\[
\dfrac{\partial \mathcal{P} (0,0)}{\partial \tau} + \sum_i \dfrac{\partial \mathcal{P} (0,0)}{\partial \tau_i} \tau_i'(0)=0.
\]
Therefore $\tau_i'(0)=0$ for each $i\in \nn$. Thus, the variational vector field to $\Gamma$ is 
\begin{equation}
\dfrac{\Gamma(\tau)}{\partial \tau}\bigg|_{\tau=0}=V+\sum_{i} \tau_i'(0) V^i+ \frac{\ptl Y_2}{\ptl \tau}(0,0)+\sum_{i} \frac{\ptl Y_2}{\ptl \tau_i}(0,0)=V.\qedhere
\end{equation}
\end{proof}

Here we show an unexpected application of Theorem~\ref{thm:intcriterion}.

\begin{example}
\label{ex:regularplaneinEngel}
An Engel structure $(E,\mh)$ is $4$-dimensional Carnot manifold where $\mh$ is a two dimensional distribution of step $3$. A representation of the Engel group $\mathbb{E}$, which is the tangent cone to each Engel structure, is given by $\rr^4$ endowed with the distribution $\mh$ generated by 
\[
X_1=\partial_{x_1} \quad \text{and} \quad X_2= \partial_{x_2} + x_ 1\partial_{x_3}+ \dfrac{x_1^2}{2} \partial_{x_4}.
\] 
The second layer is generated by
 $$X_3=[X_1,X_2]= \partial_{x_3}+ x_1 \partial_{x_4}$$
 and the third layer by $X_4=[X_1,X_3]=\partial_{x_4}$. A well-known example of horizontal singular curve, first discovered by Engel, is given by $\ga:\rr \to \rr^4$, $\ga(t)=(0,t,0,0)$. R. Bryant and L. Hsu proved in \cite{MR1240644}  that $\ga$ is rigid in the $C^1$ topology therefore this curve $\ga$ does not satisfy any geodesic equation. However  H. Sussman \cite{Sussman} proved that $\ga$ is the minimizer among all the curves whose endpoints belongs to the $x_2$-axis.
 
Let $\Omega$ be an open set in $\rr^2$ and $\Phi: \Omega \to \rr^4$ be the ruled immersion parametrized by $\Phi(u,v)=(0,u,0,v)$ whose tangent vectors are  $(X_2)_{\Phi(u,v)}$ and $(X_4)_{\Phi(u,v)}$. Then we have that the degree $\deg(\Phi(\Omega))$ is equal to four. Fix the left invariant metric $g$ that makes $X_1,\ldots,X_4$ an orthonormal basis. Taking into account equation \eqref{eq:ruledadm}, we have that  a normal vector field $V=f_3 \, X_3 + g_1 X_1 $ is  admissible if and only if
\[
\dfrac{\partial f_3 }{\partial u}= g_1,
\]
since $b_{313}=\escpr{X_3,[X_2,X_3]}=0$ and $a_{311}=\escpr{X_3,[X_2,X_1]}=-1$. Therefore $A(u,v)=(-1)$ for all $(u, v) \in \Omega$, then $A$ is linearly full in $\rr$. Thus, by Proposition~\ref{prop:linearlyfull} we gain that ruled immersion $\Phi$ is regular.

Despite the immersion $\Phi$ is foliated by singular curves that are also rigid in the $C^1$ topology,  $\Phi$ is a regular ruled immersion. Thus, by Theorem~\ref{thm:intcriterion} we obtain that each admissible vector field is integrable. Therefore it possible to compute the first variation formula \cite[Eq. (8.7), Section 8]{2019arXiv190505131C} and verify that $\Phi$ is a critical point for the area functional with respect to the left invariant metric $g$ since its mean curvature vector $\mathbf{H}_4$ of degree $4$ vanishes. Hence this plane foliated by abnormal geodesics, that do not verify any geodesic equations, satisfies the mean curvature equations for surface of degree $4$.
\end{example}

Here we show some applications  of  Theorem~\ref{thm:intcriterion} to lifted surfaces immersed of codimension 2 in an Engel structure that model the visual cortex, taking into account orientation and curvature.

\begin{example}
\label{ex:visualc}
Let $E=\rr^2  \times  \mathbb{S}^1 \times \rr$ be a smooth manifold with coordinates $p=(x,y,\theta,k)$. We set $\mathcal{H}=\text{span}\{X_1,X_2\}$, where 
\begin{equation}
 X_{1}= \cos( \theta ) \partial_{x} + \sin( \theta )  \partial_{y}+ k\partial_{\theta} \quad \text{and} \quad  X_{2}=\partial_{k}.
 \label{vector fields 1layer}
\end{equation}
The second layer is generated by
 $$X_{3}=[X_{1},X_{2}]=-\partial_{\theta}$$
 and $X_1,X_2$. The third layer by adding $ X_{4}=[X_{1},[X_{1},X_{2}]]=-\sin(\theta)\partial_{x}+ \cos(\theta) \partial_{y}$ to $X_1,\ldots,X_3$. Notice that the Carnot manifold  $(E,\mh)$ is a Engel structure. Let $\Omega$ be an open set of $\rr^2$ endowed with the Lebesgue measure. Then we consider the immersion $\Phi:\Omega \rightarrow E$, $\Phi(x,y)=(x,y,\theta(x,y),\kappa(x,y))$ where we set $\Sigma=\Phi(\Omega)$.The tangent vectors to 
$\Sigma$ are 
\begin{equation}
  \Phi_{x}=(1,0, \theta_{x}, k_{x}), \qquad \Phi_{y}=(0,1,\theta_{y},\kk_{y}).
  \label{vector fields tangent to Sigma}
\end{equation}
Following the computation in \cite[Section 4.3]{2019arXiv190505131C} the 2-vector tangent to $\Sigma$ is given 
\begin{equation}
 \begin{aligned}
 \Phi_{x}\wedge\Phi_{y}=&(\cos(\theta)\kk_{y}-\sin(\theta)\kk_{x})X_{1}\wedge X_{2}-(\cos(\theta)\theta_{y} -\sin(\theta) \theta_{x})X_{1}\wedge X_{3}\\
                       &+X_{1}\wedge X_{4}+(\theta_{x}\kk_{y}-\theta_{y}\kk_{x}-\kk(\cos(\theta)\kk_{y}-\sin(\theta)\kk_{x}))X_{2}\wedge X_{3}\\
                       &+(\sin(\theta)\kk_{y}+\cos(\theta)\kk_{x})X_{2}\wedge X_{4}
\\&+(\kk-X_1(\theta))X_{3}\wedge X_{4}.
\end{aligned}
\label{normalv2}
\end{equation}
Since the curvature is the derivative of orientation we gain that $\kappa(x,y)=X_1(\theta(x,y))$ and therefore the degree of these immersion is always equal to four. Then a tangent basis of $T_p \Sigma$ adapted to \ref{eq:tanflag} is given by
\begin{equation}
\begin{aligned}
 E_1&=\co \Phi_x+ \si \Phi_y= X_1+X_1(\kk)X_2,\\
 E_2&=-\si \Phi_x+\co \Phi_y=X_4-X_4(\theta)X_3+X_4(\kk)X_2.
\end{aligned}
\label{adapted vector tangent to Sigma}
\end{equation}
Therefore $\Sigma$ is a  FGT-$(s-3)$ ruled submanifoldruled manifold foliated by horizontal curves.
Adding $V_3=X_2-X_1(\kk)X_1$ and $V_4=X_3$ we obtain a basis of $TE$. Choosing the metric $g$ that makes $E_1,E_2,V_3,V_4$ an orthonormal basis we gain that 
\begin{align*}
a_{413}&=\escpr{V_4,[E_1,V_3]}=1+X_1(\kappa)^2,\\
b_{414}&= \escpr{V_4,[E_1,V_4]}=X_4(\theta).
\end{align*}
Therefore the admissibility system \eqref{eq:ruledadm} on the chart $\Omega$ is given by 
\[
\bar{X}_1 (f_4)= -\bar{X}_4(\theta) f_4-(1+\bar{X}_1(\theta)^2) g_3,
\] 
where $V^{\perp}=g_3 V_3+ f_4 V_4$ and the projection of the vector field $X_1$ and $X_4$ onto $\Omega$ is given by
\begin{align*}
\bar{X}_1&= \cos(\theta(x,y)) \partial_x+\sin(\theta(x,y)) \partial_y \\ 
\bar{X}_4&= -\sin(\theta(x,y))\partial_{x}+ \cos(\theta(x,y)) \partial_{y}.
\end{align*}
Notice that the matrix $A(x,y)=((1+\bar{X}_1(\theta(x,y))^2))$ never vanishes for all $(x,y) \in \Omega$, then also the matrix $\tilde{A}=DA$ defined in  Proposition~\ref{prop:inthol} never vanishes since $D(x,y)\ne0$ for all $(x,y) \in \Omega$. Therefore by Proposition~\ref{prop:linearlyfull} the surface $\Sigma$ is regular, then by Theorem~\ref{thm:intcriterion} $\Sigma$ is deformable.
\end{example}

\bibliography{degree}

\begin{thebibliography}{10}

\bibitem{AgrachevBarilariBoscain}
A.~Agrachev, D.~Barilari, and U.~Boscain.
\newblock Introduction to geodesics in sub-{R}iemannian geometry.
\newblock In {\em Geometry, analysis and dynamics on sub-{R}iemannian
  manifolds. {V}ol. {II}}, EMS Ser. Lect. Math., pages 1--83. Eur. Math. Soc.,
  Z\"{u}rich, 2016.

\bibitem{AgrachevSachkov}
A.~A. Agrachev and Y.~L. Sachkov.
\newblock {\em Control theory from the geometric viewpoint}, volume~87 of {\em
  Encyclopaedia of Mathematical Sciences}.
\newblock Springer-Verlag, Berlin, 2004.
\newblock Control Theory and Optimization, II.

\bibitem{AgracevSarychev}
A.~A. Agrachev and A.~V. Sarychev.
\newblock Abnormal sub-{R}iemannian geodesics: {M}orse index and rigidity.
\newblock {\em Ann. Inst. H. Poincar\'{e} Anal. Non Lin\'{e}aire},
  13(6):635--690, 1996.

\bibitem{AGS13}
L.~Ambrosio, N.~Gigli, and G.~Savar\'{e}.
\newblock Density of {L}ipschitz functions and equivalence of weak gradients in
  metric measure spaces.
\newblock {\em Rev. Mat. Iberoam.}, 29(3):969--996, 2013.

\bibitem{Bellaiche}
A.~Bella{\"{\i}}che.
\newblock {The tangent space in sub-{R}iemannian geometry}.
\newblock {\em J. Math. Sci. (New York)}, 83(4):461--476, 1997.
\newblock Dynamical systems, 3.

\bibitem{Brezis}
H.~Brezis.
\newblock {\em Functional analysis, {S}obolev spaces and partial differential
  equations}.
\newblock Universitext. Springer, New York, 2011.

\bibitem{BL74}
J.~K. Brooks and P.~W. Lewis.
\newblock Linear operators and vector measures.
\newblock {\em Trans. Amer. Math. Soc.}, 192:139--162, 1974.

\bibitem{MR924805}
R.~L. Bryant.
\newblock {On notions of equivalence of variational problems with one
  independent variable}.
\newblock In {\em {Differential geometry: the interface between pure and
  applied mathematics ({S}an {A}ntonio, {T}ex., 1986)}}, volume~68 of {\em
  {Contemp. Math.}}, pages 65--76. Amer. Math. Soc., Providence, RI, 1987.

\bibitem{MR1083148}
R.~L. Bryant, S.~S. Chern, R.~B. Gardner, H.~L. Goldschmidt, and P.~A.
  Griffiths.
\newblock {\em Exterior differential systems}, volume~18 of {\em Mathematical
  Sciences Research Institute Publications}.
\newblock Springer-Verlag, New York, 1991.

\bibitem{MR1240644}
R.~L. Bryant and L.~Hsu.
\newblock {Rigidity of integral curves of rank {$2$} distributions}.
\newblock {\em Invent. Math.}, 114(2):435--461, 1993.

\bibitem{MR1190006}
E.~Cartan.
\newblock {\em Le\c{c}ons sur la g\'{e}om\'{e}trie projective complexe. {L}a
  th\'{e}orie des groupes finis et continus et la g\'{e}om\'{e}trie
  diff\'{e}rentielle trait\'{e}es par la m\'{e}thode du rep\`ere mobile.
  {L}e\c{c}ons sur la th\'{e}orie des espaces \`a connexion projective}.
\newblock Les Grands Classiques Gauthier-Villars. [Gauthier-Villars Great
  Classics]. \'{E}ditions Jacques Gabay, Sceaux, 1992.
\newblock Reprint of the editions of 1931, 1937 and 1937.

\bibitem{Cheeger99}
J.~Cheeger.
\newblock Differentiability of {L}ipschitz functions on metric measure spaces.
\newblock {\em Geom. Funct. Anal.}, 9(3):428--517, 1999.

\bibitem{2019arXiv190204015C}
G.~{Citti}, G.~{Giovannardi}, and M.~{Ritor{\'e}}.
\newblock {Variational formulas for curves of fixed degree}.
\newblock {\em arXiv e-prints}, page arXiv:1902.04015, Feb 2019.

\bibitem{2019arXiv190505131C}
G.~{Citti}, G.~{Giovannardi}, and M.~{Ritor{\'e}}.
\newblock {Variational formulas for submanifolds of fixed degree}.
\newblock {\em arXiv e-prints}, page arXiv:1905.05131, May 2019.

\bibitem{CS06}
G.~Citti and A.~Sarti.
\newblock {A cortical based model of perceptual completion in the
  roto-translation space}.
\newblock {\em J. Math. Imaging Vision}, 24(3):307--326, 2006.

\bibitem{cittilibro}
G.~Citti and A.~Sarti, editors.
\newblock {\em {Neuromathematics of vision}}.
\newblock {Lecture Notes in Morphogenesis}. Springer, Heidelberg, 2014.

\bibitem{CS15}
G.~Citti and A.~Sarti.
\newblock Models of the visual cortex in {L}ie groups.
\newblock In {\em Harmonic and geometric analysis}, Adv. Courses Math. CRM
  Barcelona, pages 1--55. Birkh\"{a}user/Springer Basel AG, Basel, 2015.

\bibitem{Carmo}
M.~P. do~Carmo.
\newblock {\em {Riemannian geometry}}.
\newblock {Mathematics: Theory \& Applications}. Birkh{\"a}user Boston, Inc.,
  Boston, MA, 1992.
\newblock Translated from the second Portuguese edition by Francis Flaherty.

\bibitem{DobbinsZucker}
A.~Dobbins, S.~W. Zucker, and M.~S. Cynader.
\newblock {Endstopped neurons in the visual cortex as a substrate for
  calculating curvature}.
\newblock {\em Nature}, 329(6138):438--441, 1987.

\bibitem{EF17}
E.~Feleqi and F.~Rampazzo.
\newblock Iterated {L}ie brackets for nonsmooth vector fields.
\newblock {\em NoDEA Nonlinear Differential Equations Appl.}, 24(6):Art. 61,
  43, 2017.

\bibitem{Folland84}
G.~B. Folland.
\newblock {\em Real analysis}.
\newblock Pure and Applied Mathematics (New York). John Wiley \& Sons, Inc.,
  New York, 1984.
\newblock Modern techniques and their applications, A Wiley-Interscience
  Publication.

\bibitem{Gigli18}
N.~Gigli.
\newblock Nonsmooth differential geometry---an approach tailored for spaces
  with {R}icci curvature bounded from below.
\newblock {\em Mem. Amer. Math. Soc.}, 251(1196):v+161, 2018.

\bibitem{MR684663}
P.~A. Griffiths.
\newblock {\em {Exterior differential systems and the calculus of variations}},
  volume~25 of {\em {Progress in Mathematics}}.
\newblock Birkh{\"a}user, Boston, Mass., 1983.

\bibitem{Gromov86}
M.~Gromov.
\newblock {\em Partial differential relations}, volume~9 of {\em Ergebnisse der
  Mathematik und ihrer Grenzgebiete (3) [Results in Mathematics and Related
  Areas (3)]}.
\newblock Springer-Verlag, Berlin, 1986.

\bibitem{Gromov}
M.~Gromov.
\newblock {Carnot-{C}arath{\'e}odory spaces seen from within}.
\newblock In {\em {Sub-{R}iemannian geometry}}, volume 144 of {\em {Progr.
  Math.}}, pages 79--323. Birkh{\"a}user, Basel, 1996.

\bibitem{Hartman}
P.~Hartman.
\newblock {\em {Ordinary differential equations}}, volume~38 of {\em {Classics
  in Applied Mathematics}}.
\newblock Society for Industrial and Applied Mathematics (SIAM), Philadelphia,
  PA, 2002.
\newblock Corrected reprint of the second (1982) edition [Birkh{\"a}user,
  Boston, MA; MR0658490 (83e:34002)], With a foreword by Peter Bates.

\bibitem{HKST15}
J.~Heinonen, P.~Koskela, N.~Shanmugalingam, and J.~T. Tyson.
\newblock {\em Sobolev spaces on metric measure spaces}, volume~27 of {\em New
  Mathematical Monographs}.
\newblock Cambridge University Press, Cambridge, 2015.
\newblock An approach based on upper gradients.

\bibitem{HH91}
H.~Hermes.
\newblock Nilpotent and high-order approximations of vector field systems.
\newblock {\em SIAM Rev.}, 33(2):238--264, 1991.

\bibitem{MR1189496}
L.~Hsu.
\newblock {Calculus of variations via the {G}riffiths formalism}.
\newblock {\em J. Differential Geom.}, 36(3):551--589, 1992.

\bibitem{HW62}
D.~H. Hubel and T.~N. Wiesel.
\newblock Receptive fields, binocular interaction and functional architecture
  in the cat's visual cortex.
\newblock {\em The Journal of Physiology}, 160(1):106--154, 1962.

\bibitem{Extremalcurves}
E.~Le~Donne, G.~P. Leonardi, R.~Monti, and D.~Vittone.
\newblock Extremal curves in nilpotent {L}ie groups.
\newblock {\em Geom. Funct. Anal.}, 23(4):1371--1401, 2013.

\bibitem{DonneMagnani}
E.~{Le Donne} and V.~Magnani.
\newblock {Measure of submanifolds in the {E}ngel group}.
\newblock {\em Rev. Mat. Iberoam.}, 26(1):333--346, 2010.

\bibitem{LeonardiMonti}
G.~P. Leonardi and R.~Monti.
\newblock End-point equations and regularity of sub-{R}iemannian geodesics.
\newblock {\em Geom. Funct. Anal.}, 18(2):552--582, 2008.

\bibitem{Maggi12}
F.~Maggi.
\newblock {\em Sets of finite perimeter and geometric variational problems},
  volume 135 of {\em Cambridge Studies in Advanced Mathematics}.
\newblock Cambridge University Press, Cambridge, 2012.
\newblock An introduction to geometric measure theory.

\bibitem{vittonemagnani}
V.~Magnani and D.~Vittone.
\newblock {An intrinsic measure for submanifolds in stratified groups}.
\newblock {\em Journal f{\"u}r die reine und angewandte Mathematik (Crelles
  Journal)}, 2008(619):203--232, 2008.

\bibitem{Mont94a}
R.~Montgomery.
\newblock Abnormal minimizers.
\newblock {\em SIAM J. Control Optim.}, 32(6):1605--1620, 1994.

\bibitem{Mont94b}
R.~Montgomery.
\newblock Singular extremals on {L}ie groups.
\newblock {\em Math. Control Signals Systems}, 7(3):217--234, 1994.

\bibitem{Montgomery}
R.~Montgomery.
\newblock {\em {A tour of subriemannian geometries, their geodesics and
  applications}}, volume~91 of {\em {Mathematical Surveys and Monographs}}.
\newblock American Mathematical Society, Providence, RI, 2002.

\bibitem{Monti}
R.~Monti.
\newblock The regularity problem for sub-{R}iemannian geodesics.
\newblock In {\em Geometric control theory and sub-{R}iemannian geometry},
  volume~5 of {\em Springer INdAM Ser.}, pages 313--332. Springer, Cham, 2014.

\bibitem{NCS18}
N.~{Montobbio}, A.~{Sarti}, and G.~{Citti}.
\newblock {A metric model for the functional architecture of the visual
  cortex}.
\newblock {\em arXiv e-prints}, page arXiv:1807.02479, Jul 2018.

\bibitem{Pansu16}
P.~{Pansu}.
\newblock {Submanifolds and differential forms on Carnot manifolds, after M.
  Gromov and M. Rumin}.
\newblock {\em arXiv e-prints}, page arXiv:1604.06333, Apr 2016.

\bibitem{Petitot2014}
J.~Petitot.
\newblock {\em Landmarks for Neurogeometry}, pages 1--85.
\newblock Springer Berlin Heidelberg, Berlin, Heidelberg, 2014.

\bibitem{PT01}
P.~Piccione and D.~V. Tausk.
\newblock On the {B}anach differential structure for sets of maps on
  non-compact domains.
\newblock {\em Nonlinear Anal.}, 46(2, Ser. A: Theory Methods):245--265, 2001.

\bibitem{Rifford}
L.~Rifford.
\newblock {\em Sub-{R}iemannian geometry and optimal transport}.
\newblock SpringerBriefs in Mathematics. Springer, Cham, 2014.

\bibitem{Rudin}
W.~Rudin.
\newblock {\em Real and complex analysis}.
\newblock McGraw-Hill Book Co., New York, third edition, 1987.

\bibitem{S70}
I.~Singer.
\newblock {\em Bases in {B}anach spaces. {I}}.
\newblock Springer-Verlag, New York-Berlin, 1970.
\newblock Die Grundlehren der mathematischen Wissenschaften, Band 154.

\bibitem{Sussman}
H.~J. Sussmann.
\newblock {A cornucopia of four-dimensional abnormal sub-{R}iemannian
  minimizers}.
\newblock In {\em {Sub-{R}iemannian geometry}}, volume 144 of {\em {Progr.
  Math.}}, pages 341--364. Birkh{\"a}user, Basel, 1996.

\end{thebibliography}

\end{document}